\documentclass{article}

\usepackage{amsmath,amssymb,color,amsthm}
\usepackage{enumerate,enumitem}
\usepackage{graphicx,subfigure}
\usepackage{verbatim,mathrsfs}
\usepackage[margin=1.5in]{geometry}

\usepackage{authblk}

\DeclareGraphicsExtensions{.pdf}

\usepackage[usenames,dvipsnames]{xcolor}
\theoremstyle{plain}
\newtheorem{theorem}{Theorem}
\newtheorem{lemma}[theorem]{Lemma}
\newtheorem{corollary}[theorem]{Corollary}
\newtheorem{conjecture}[theorem]{Conjecture}

\newtheorem*{claim*}{Claim}
\newtheorem{claim}{Claim}

\newtheorem{observation}[theorem]{Observation}
\newtheorem{question}[theorem]{Question}
\newtheorem{remark}[theorem]{Remark}
\newtheorem*{remark*}{Remark}

\theoremstyle{definition}

\newcommand{\eps}{\varepsilon}

\newcommand{\cR}{\mathcal{R}}

\newcommand{\NN}{\mathbb{N}}

\newcommand{\HH}{\mathcal{H}}

\newcommand{\girth}{{\rm girth}}
\newcommand{\req}{\overset{R}{\sim}}
\newcommand{\epsto}{\overset{\epsilon}{\to}}

\newcommand{\erdoes}{Erd{\H{o}}s}

\newcommand{\lovasz}{Lov{\'a}sz}

\title{Conditions on Ramsey non-equivalence}
\author{Maria Axenovich}
\author{Jonathan Rollin}
\author{Torsten Ueckerdt}
\affil{Department of Mathematics, Karlsruhe Institute of Technology}

\begin{document}

\maketitle

\begin{abstract}
Given a graph $H$, a graph $G$ is  called a \emph{Ramsey graph of $H$} if there is a monochromatic copy of $H$ in every coloring of the edges of $G$ with two colors.
Two graphs $G$, $H$ are called \emph{Ramsey equivalent} if they have the same set of Ramsey graphs.
Fox \textit{et al.}
% ~\cite{Fox_EquivClique}
[{\em J. Combin. Theory Ser. B} \textbf{109} (2014), 120--133]
% Szabo \textit{et al.}
% ~\cite{Szabo_RamseyMinBipartite}
% [{\em J. Graph Theory} \textbf{64}(2) (2010), 150--164]
asked whether there are two non-isomorphic connected graphs that are  Ramsey equivalent.
They proved that a clique is not Ramsey equivalent to any other connected graph.  
Results of Ne{\v{s}}et{\v{r}}il \textit{et al.} showed that any two graphs with different clique number
% ~\cite{SimpleRamseySameClique}
[{\em Combinatorica} \textbf{1}(2) (1981), 199--202]
or different odd girth
% ~\cite{RamseyOddGirth}
[{\em Comment. Math. Univ. Carolin.} \textbf{20}(3) (1979), 565--582]
are not Ramsey equivalent. 
These are the only structural graph parameters we know that ``distinguish'' two graphs in the above sense. 
This paper provides further supportive evidence for a negative answer to the question of Fox \textit{et al.} by claiming that for wide classes of graphs, chromatic number is a distinguishing parameter. 
In addition, it is shown here that all stars and paths and all connected graphs on at most $5$ vertices are not Ramsey equivalent to any other connected graph.
Moreover two connected graphs are not Ramsey equivalent if they belong to  a special class of trees or to classes of graphs with clique-reduction properties.
\end{abstract}


\section{Introduction}

Given a graph $H$, a graph $G$ is called  a \emph{Ramsey graph of $H$} if there is a monochromatic copy of $H$ in every coloring of the edges of $G$ with  two colors.
 If $G$ is a Ramsey graph of $H$, we write $G\to H$ and say that $G$ arrows $H$.
 We denote by $\cR(H)$ the set of all graphs that arrow $H$, and call it the {\em Ramsey class of $H$}.
 So the Ramsey number $R(H)$ is the smallest integer $n$ such that $K_n \in \cR(H)$, where $K_n$ denotes the complete graph on $n$ vertices.
 Two graphs $G$, $H$ are called {\em Ramsey equivalent} if they have the same Ramsey class.
 We write $G\req H$ if $G$ is Ramsey equivalent to $H$, and write $G\not \req H$ otherwise.
 The study of Ramsey classes was initiated by the fundamental work of Ne{\v{s}}et{\v{r}}il and R\"odl~\cite{CliqueNoRamseyGeneral} and Burr, Erd\H{o}s, and Lov\'asz~\cite{BEL_Ramsey}.
 However,  the notion of Ramsey equivalence of graphs was raised only recently by Szab{\'o} \textit{et al.}~\cite{Szabo_RamseyMinBipartite}. 
 It was shown in~\cite{Szabo_RamseyMinBipartite}, that there are two non-isomorphic  graphs that are Ramsey equivalent, for example $ G \req H$, where $G=K_t$ and $H$ is a vertex disjoint union of $K_t$ and $K_{2}$ for $t\geq 4$.
%  Here $K_t$ is the complete graph on $t$ vertices.
 In this example $H$ is a disconnected graph. Fox \textit{et al.} formulated a question for connected graphs:
 \begin{question}[\cite{Fox_EquivClique}]\label{que:MAIN}
  Are there two non-isomorphic connected graphs $G$ and $H$ with $G\req H$?
 \end{question}
 
 Note that $G \not\req H$ if and only if there exists a graph $\Gamma$ such that $\Gamma \to H$ and $\Gamma\not\to G$ or $\Gamma \to G$ and $\Gamma \not\to H$.
 In this case we call $\Gamma$ a graph, {\em distinguishing} $G$ and $H$.
 So in order to prove that $G \not\req H$, it is sufficient to explicitly construct a distinguishing graph.
 Another approach is to identify a graph parameter $s$, such that $s(G)\neq s(H)$ implies that $G\not\req H$.
 In this case, we say that $s$ is a {\it Ramsey distinguishing parameter}.
 The only structural graph parameters that we know to be Ramsey distinguishing are the clique number $\omega$ and the odd girth $g_o$, where $\omega$ is the largest number of vertices in a clique of the graph, and $g_o$ is the length of its shortest odd cycle. 
 Specifically, it is shown in~\cite{SimpleRamseySameClique, RamseyOddGirth} that if $\omega(H)=\omega$ and $g_o(H)=g_o$ then there are  Ramsey graphs $G, G'\in \cR(H)$ such that $\omega(G)=\omega$ and $g_o(G')= g_o$.
 Note that if Question~\ref{que:MAIN} has a negative answer, then any graph parameter is a Ramsey distinguishing parameter for the class of connected graphs.\\
 
 In this paper, we provide a supporting evidence for a 'No'-answer to Question~\ref{que:MAIN} by the following theorems, focusing  on another graph parameter, the chromatic number, $\chi$. 

\begin{observation}\label{obs::bipartite}
 If $G$ and $H$ are graphs, $\chi(G)=2$ and $\chi(H)>2$ then $G\not\req H$.
\end{observation}
 Indeed, a sufficiently large complete bipartite graph arrows any fixed bipartite graph~\cite{bipartiteRamsey}.
 But it does not contain, and thus does not arrow any non-bipartite graph.
 Here, we prove that for several large classes of connected graphs, the chromatic number is a Ramsey distinguishing parameter. 
 A graph is called \emph{clique-splittable} if its vertex set can be partitioned into two subsets, each inducing a subgraph of smaller clique number.
 Note that any graph $G$ with $\chi(G) \leq 2\omega(G)-2$ is clique-splittable.
 In particular all cliques and all planar graphs containing a triangle are clique-splittable.
 The triangle-free clique-splittable graphs are precisely the bipartite graphs.
 
 \begin{theorem}\label{thm::DifferentChrNo}
%  **thm::DifferentChrNo**\\
 If $G,H$ are graphs, $G$ is clique-splittable and $\chi(G) < \chi(H)$, then $G\not\req H$.
 \end{theorem}

\begin{corollary}\label{Cor}
% **Cor**\\
 If $G,H$ are graphs, $\chi(G)\leq 2\omega(G)-2$ and $\chi(G)\neq \chi(H)$, then~$G\not\req H$.
\end{corollary}

% 
% \begin{theorem}
% *** Can we actually prove this?  If yes, we need to move $\cR_s$ definition up here. **** \\
% There are two graphs $G$ and $H$ such that $\chi(G)\neq \chi(H)$ but $\cR_\chi(G)=\cR_\chi(H)$.
% \end{theorem}

Theorem~\ref{thm::DifferentChrNo} distinguishes pairs of graphs of distinct chromatic number under some splittability condition.
The following theorem requires stronger assumptions but also applies to graphs of the same chromatic number.
If $H$ is a subgraph (proper subgraph) of $G$, we write $H\subseteq G$  ($H\subsetneq G$).
 
 \begin{theorem}\label{thm::SameChrNo}
%  **thm::SameChrNo**\\
 Let a connected graph $G$ satisfy the following two properties:
 \begin{enumerate}[label=\upshape\textbf{\arabic{enumi})}]
  \item There is an independent set $S\subset V(G)$ such that $\omega(G - S) < \omega(G)$.
  \item There is a proper $\chi(G)$-vertex-coloring of $G$ in which some two color classes induce a subgraph of a matching.
 \end{enumerate}
Let   $H$ be a connected graph, not isomorphic to $G$,   such that either  $H\subsetneq G$ or   $\chi(H) \geq  \chi(G)$. 
 Then  $G\not\req H$.
 \end{theorem}
 
 In Theorems~\ref{thm::DifferentChrNo} and~\ref{thm::SameChrNo} we distinguish pairs of graphs under certain properties.
 Call a graph $G$ \emph{Ramsey isolated} if $G\not\req H$ for any connected graph $H$ not isomorphic to $G$.
 Note that Question~\ref{que:MAIN} asks whether every connected graph is Ramsey isolated or not.
 We apply the previous results to identify large families of Ramsey isolated graphs. 
 The $k$-wheel is the graph on  $k+1$ vertices obtained from a cycle of length $k$ by adding a vertex adjacent to all vertices of the cycle. 
 
 \begin{theorem}\label{thm::PathStar}{\ }
%  **thm::PathStar**\\
  \begin{enumerate}
   \item{} If $G$ is connected, $\chi(G)=\omega(G)$ and there is a proper $\chi(G)$-vertex-coloring of $G$ in which some two color classes induce a subgraph of a  matching in $G$, then $G$ is Ramsey isolated.
  \item{} Any path and any star are Ramsey isolated.
    \item{} Each connected graph on at most $5$ vertices is Ramsey isolated.
   \end{enumerate}
 \end{theorem}

%  \begin{theorem}\label{thm::smallGraphs}
%   Any two connected non-isomorphic graphs on at most $5$ vertices are not Ramsey equivalent.
%  \end{theorem}

 \begin{remark}\label{smallDistinguishingGraphs}
  If $F$ distinguishes $G$ and $H$ then $F$ has at least $\min\{R(G),R(H)\}$ vertices.
  The distinguishing graphs used in the proof of Theorem~\ref{thm::PathStar} are rather large, except for stars.
  However, for most pairs $G,H$ of distinct connected graphs on at most $5$ vertices there is a distinguishing graph on $\min\{R(G),R(H)\}$ vertices.
 \end{remark}

 A tree $T$ on $k$ vertices is called \textit{balanced} if deleting some edge splits $T$ into components of order at most $\lceil \frac{k+1}{2} \rceil$ each.
 The extremal function ${\rm ex}(n,H)$ is the largest number of edges in an $n$-vertex graph with no copy of $H$.
 The Erd\H{o}s-S\'os-Conjecture states that ${\rm ex}(n,T)\leq \frac{k-2}{2}n$ for any tree $T$ on $k$ vertices.
 We remark that recently, Ajtai, Koml\'os, Simonovits and Szemer\'edi announced a proof of the conjecture for large $k$~\cite{Simonovits1,Simonovits2,Simonovits3}.
 We state here a much weaker conjecture:
 
\begin{conjecture} \label{tree-conjecture}
 There is a positive $\epsilon$ and an integer $n_\epsilon$ such that  ${\rm ex} (n, T) \leq \frac{k-1 -\epsilon}{2} n$ for any tree on $k$ vertices and  $n > n_\epsilon$.
%  ${\rm ex}(n, T) \leq  \frac{k-1 -\epsilon}{2} n$, for any tree on $k$ vertices and any $n> n_\epsilon$.
\end{conjecture}

\begin{theorem}\label{thm::Trees}
 If Conjecture~\ref{tree-conjecture} is true then any two trees of different order are not Ramsey equivalent.
 If $T_k$ is a balanced tree on $k$ vertices and $T_{\ell}$ is any tree on $\ell\geq k+1$ vertices, then  $T_k \not\req T_{\ell}$.
 \end{theorem}

 The next theorem makes use of multicolor Ramsey numbers. Here $R(G_1, G_2, G_3)$ is the minimum integer $n$ such that any coloring of the edges of $K_n$ in red, blue, and green  has a red $G_1$, a blue $G_2$, or a green $G_3$.
 We write $G\not\req_k H$ if $G$ and $H$ are not Ramsey equivalent in $k$ colors, i.e., there is a graph $\Gamma$ such that any $k$-coloring of its edges contains a monochromatic $G$  (we write $\Gamma\to_k G$)  and there is such a coloring avoiding monochromatic $H$, or vice versa.

 \begin{theorem}\label{thm::MoreColors}
If $G$ and $H$ are graphs then $G\not\req H$ if one of the following conditions holds:
\begin{itemize}
 \item There is a graph $F$ such that $R(G,G,F) < R(H,H,F)$.
 \item $G\subseteq H$ and there is $k\geq 2$ with $G\not\req_k H$.
\end{itemize}

 \end{theorem}
 
%  \jonathan{Do we need connectedness for Theorem~\ref{thm::MoreColors}?}
%  \torsten{No.}
 
 The paper is structured as follows.
 In Section~\ref{sec::overview} we provide a more detailed summary of known results on Ramsey equivalence, as well as some observations.
 In Section~\ref{sec::Lemmas} we give known and our preliminary results that will be used in proving the main theorems.
 Section~\ref{Proofs} contains the proofs of the main results, the Appendix provides lemmas used to prove Remark~\ref{smallDistinguishingGraphs} and other routine lemmas.
%  The proof of Theorem~\ref{thm::PathStar}.3 uses Theorem~\ref{thm::smallGraphs}, which in turn uses only Theorem~\ref{thm::PathStar}.1 and Theorem~\ref{thm::PathStar}.2.
%  However for the sake of better structure, we present the proofs of Theorems~\ref{thm::PathStar}.1 - \ref{thm::PathStar}.3 and Theorem~\ref{thm::smallGraphs} in the order of their numbering.
 Finally, Section~\ref{Conclusions} contains conclusions and open questions.
 We refer the reader to~\cite{West} for all standard notations in graph theory and to Figure~\ref{fig::5Vert} for notations used for small graphs.
 We omit floors and ceilings as long as the meaning is clear from context.
 We also assume that the graphs under consideration have at least one edge.

%%%%%%%%%%%%%%%%%%%%%%%%%%%%%%%%%%%%%%%%%%%
%%%%%%%%   General Observations   %%%%%%%%%
%%%%%%%%%%%%%%%%%%%%%%%%%%%%%%%%%%%%%%%%%%%

\section{Overview and Observations}\label{sec::overview}

 As mentioned in the introduction, two connected non-isomorphic graphs $G$ and $H$ are not Ramsey equivalent if one of them is a clique~\cite{Fox_EquivClique}, or if $\omega(G)\neq \omega(H)$~\cite{SimpleRamseySameClique} or when $g_o(G)\neq g_o(H)$~\cite{RamseyOddGirth}.
 Note that it is an open question whether $\max\{\girth(F): F\to G\}=\girth(G)$ for all graphs $G$~\cite{Nesetril_HighGirthChrNumber}, where $\girth(G)$ is the length of a shortest cycle in $G$.
 There are several other results about Ramsey classes that are useful in checking whether some two given graphs are Ramsey equivalent or not.
 Note that two graphs are Ramsey equivalent if and only if they have the same set of minimal Ramsey graphs.
 A graph $\Gamma$ is {\em minimal Ramsey} for $H$ if $\Gamma \to H$, but $\Gamma'\not\to H$ for any proper subgraph $\Gamma'$ of $\Gamma$. 
 We need to define a few graphs to state the known results: $P_n$, $C_n$ is the path, cycle on $n$ vertices, respectively,  $H_{t,d}$ is a graph on $t+1$ vertices such that one vertex has degree $d$ and the other vertices induce $K_t$, $K_{a,b}$ is the complete bipartite graph with parts of sizes $a$ and $b$, respectively.
 It was shown in~\cite{RamseyWithoutBipartite} that if $H$ does not contain $K_{a,b}$, $a,b\geq 3$ or $1=a\leq b\leq 2$, then there is $\Gamma\in \cR(H)$ such that $\Gamma$ does not contain $K_{a,b}$.
 In general, for an integral graph parameter $s$, we define
 \[
  \cR_s = \min\{s(\Gamma): \Gamma\text{ is minimal Ramsey for } H\}.
 \]
 If $\cR_s(H) \neq \cR_s(G)$ then clearly $G\not\req H$. 
 When $s=\delta$, the minimum degree, a number of results have been obtained:  $\cR_\delta(K_t)=(t-1)^2$,~\cite{BEL_Ramsey, Fox_MinDegreeRamseyMinimal},   $\cR_\delta(H_{t,1}) = t-1$,~\cite{Fox_EquivClique},   $\cR_\delta( H_{t,d}) = d^2$, for $2 \leq d \leq t$,~\cite{Grinshpun_RamsesyMinimal}, 
 $\cR_\delta(K_{a,b}) = 2 \min\{a,b\} -1$,~\cite{Fox_MinDegreeRamseyMinimal},  $\cR_\delta(H) =1$ if $H$ is a tree,~\cite{Szabo_RamseyMinBipartite}, or when $H$ is $K_{t, t}$ plus a pending edge,~\cite{Fox_EquivClique},  $\cR_\delta(C_n)=3$ for even $n\geq 4$,~\cite{Szabo_RamseyMinBipartite}.  
 Burr, \erdoes,~\lovasz~\cite{BEL_Ramsey} conjectured  that for every integer $\chi$ there is $H$ with $\chi(H)=\chi$ and $R_\chi(H)=(\chi-1)^2+1$.
 This conjecture has been proven by Zhu~\cite{zhu_fractionalHedetniemi}.
 Burr \textit{et al.} proved that $R_{\chi}(G)\leq R(G)$ \cite{BEL_Ramsey}.
%  Hence $G\not\req H$ if $G$ is an $n$-vertex graph and $\chi(H) > 4^{n}$ because $R_{\chi}(H) \geq \chi(H) > 4^{n} \geq R(G) \geq R_{\chi}(G)$.
 In~\cite{DegreeRamseyTrees} the authors asked whether $\cR_\Delta(H)$ is bounded by a function depending on maximum degree $\Delta(H)$ only.
 They prove that if $H$ is a tree, then $2\Delta(H) - 1 \leq \cR_\Delta(H) \leq 4(\Delta(H)-1)$.  
 We observe that for a non-bipartite $H$,   $\cR_\Delta(H)\geq 2\Delta(H)$, see Lemma~\ref{lem::maxDegree}.
  In~\cite{DensityThreshold} a threshold $t_H$ is given, such that almost all graphs with density larger than $t_H$ are Ramsey for $H$ and almost all with smaller density are not.
  The introduction of~\cite{RamseyInfinite} covers several results on   graphs with  infinitely many minimal Ramsey graphs.
 
 It remains unclear whether a connected graph could or could not be Ramsey equivalent to its subgraph.
 Note that an edge-transitive graph $H$ is not Ramsey equivalent to any of its subgraphs because coloring the edges of $\Gamma-e$ for a Ramsey minimal graph $\Gamma$ of $H$ in two colors, gives a monochromatic copy of $H-e'$ for some edge $e'$.
 Since $H-e'$ is isomorphic to $H-e''$ for any two edges, $H-e'$ contains any proper subgraph $H'$ of $H$. We see that $\Gamma-e \to H'$, but $\Gamma-e\not\to H$.

%%%%%%%%%%%%%%%%%%%%%%%%%%%%%%%%%%%%%%%%%%%
%%%%%%%%   Clique-Splittable   %%%%%%%%%%%%%%
%%%%%%%%%%%%%%%%%%%%%%%%%%%%%%%%%%%%%%%%%%%

\section{Preliminary Lemmas}\label{sec::Lemmas}

The following lemma is an easy generalization of the Focusing Lemma in~\cite{Fox_EquivClique}.

 \begin{lemma}[Focusing Lemma,~\cite{Fox_EquivClique}]\label{lem::focusing}
 Let $(A\cup  B, E)$ be a bipartite graph with a $2$-edge-coloring. Then there is a subset $B'\subseteq B$, $|B'|\geq |B|/2^{|A|}$, such that for each $a\in A$ all edges from $a$ to $B'$ are of the same color.
 \end{lemma}

% \begin{lemma}[Focusing Lemma,~\cite{Fox_EquivClique}]\label{lem::focusing}{\ \\}
% Let $Q= (A\cup  B, E)$ be a bipartite graph with a $2$-edge-coloring. Then there are subsets
% $A'\subseteq A$ and $B'\subseteq B$, $|A'|\geq |A|/2$ and $|B'|\geq |B|/2^{|A|}$, such that 
% \begin{itemize}
% \item{} for each $a\in A$ all edges from $a$ to $B'$ are of the same color,
% \item{} all edges between $A'$ and $B'$ are of the same color. 
% \end{itemize}
% \end{lemma}
 
\begin{lemma}[\cite{SimpleRamseySameClique}]\label{lem::RamseyGraphSameCliqueNo}
For any  graph $G$ there is a   graph $F\in\cR(G)$ with $\omega(G)=\omega(F)$. 
\end{lemma} 
 
We write $V(\HH)$ and $E(\HH)$ for the vertex set, respectively the edge set, of a graph or hypergraph $\HH$.
A hypergraph is $k$-uniform if every hyperedge has size $k$.
The girth of a hypergraph is the smallest number $k\geq 2$ of distinct vertices $v_0,\ldots, v_{k-1}$ and distinct hyperedges $E_0,\ldots, E_{k-1}$ such that for each $i$, $0\leq i\leq k-1$, $\{v_i,v_{i+1}\}\subseteq E_i$ (indices taken modulo $k$).
The independence number of a hypergraph is the size of a largest set of vertices which does not contain a hyperedge completely.
The chromatic number of a hypergraph is the smallest number of colors in a proper vertex coloring, i.e., a coloring without monochromatic hyperedges.
 
\begin{lemma}[\cite{hypergraphSmallIndependence}]\label{lem::HypergraphGirthIndependence}
For any integers $k,g\geq 2$ and any  $\epsilon>0$ there is an integer $n$ and  a $k$-uniform hypergraph on $n$ vertices  with  girth at least $g$ and independence number less than $\epsilon n$.  
\end{lemma} 

From this lemma one easily derives the following well-known result.

\begin{lemma}[\cite{hypergraphSmallIndependence}]\label{lem::HypergraphGirthChromatic}
For any integers $k,g, \chi \geq 2$  there is a $k$-uniform hypergraph    with  girth at least $g$ and chromatic number at least $\chi$.
\end{lemma} 
\begin{proof}
 Let $\epsilon < \frac{1}{\chi}$ and let $\HH$ denote a $k$-uniform hypergraph with girth at least $g$ and independence number at most $\epsilon |V(\HH)|$, which exists by Lemma~\ref{lem::HypergraphGirthIndependence}.
 In any $\chi$-coloring of $V(\HH)$ there is a set of at least $\frac{1}{\chi}|V(\HH)| > \epsilon|V(\HH)|$ vertices of the same color.
 Hence this color class induces an edge of $\HH$.
 Thus the coloring is not proper and $\chi(\HH) > \chi$.
\end{proof}

For graphs $F$, $G$ and for $\epsilon>0$ we write $F\epsto G$ if for any set $S\subseteq V(F)$ with $|S|\geq \epsilon |V(F)|$, we have $F[S]\to G$.
Here $F[S]$ denotes the subgraph of $F$ induced by the vertices in $S$.

\begin{lemma}[\cite{Fox_EquivClique}]\label{lem::epsRamsey}
 For any $\epsilon>0$ and any graph $H$, there is a graph $F$ with $\omega(F)=\omega(H)$ and $F\epsto H$.
\end{lemma}

\begin{proof}
 Let $F'$ be a graph such that $F'\to H$ and $\omega(F')=\omega(H)$.  Such a  graph exists by Lemma~\ref{lem::RamseyGraphSameCliqueNo}.
 Further let $\HH$ denote a $|V(F')|$-uniform hypergraph of girth at least $4$ and no independent set of size $\epsilon |V(\HH)|$, which exists by Lemma~\ref{lem::HypergraphGirthIndependence}.
 Construct a graph $F$ by placing a copy of $F'$ on the vertices of each hyperedge of $\HH$.
 Then $F$ is a   graph on $|V(\HH)|$ vertices with $\omega(F)=\omega(F')=\omega(H)$. 
Each vertex set of size at least $\epsilon |V(\HH)|$ induces a hyperedge in $\HH$ and thus a copy of $F'$ in $F$ which arrows $H$.
\end{proof}

We have the following corollary, since any graph which arrows $H$ contains $H$.

\begin{lemma}\label{lem::SparseGraphRichCopies}
 For any $\epsilon>0$ and any graph $H$, there is a graph $F$ with $\omega(F)= \omega(H)$ and each set of $\epsilon|V(F)|$ vertices in $F$ containing a copy of $H$.
\end{lemma}
% \begin{proof}
%  Consider a $|V(H)|$-uniform, $n$-vertex hypergraph of girth at least $4$ and no independent set on $\epsilon n$ vertices. Such exists by Lemma~\ref{lem::HypergraphGirthIndependence}.
%  Construct  a graph $F$ by placing a copy of $H$ on the vertices of each hyperedge.
%  Then $\omega(F) = \omega(H)$.  Moreover every set of at least $\epsilon n$ vertices induces a hyperedge and thus a copy of $H$.
% \end{proof}

 \begin{lemma}\label{lem::maxDegree}
 If a graph $H$ is not bipartite then $\cR_\Delta(H)\geq 2\Delta(H)$. The lower bound is tight. 
\end{lemma}
\begin{proof}
 Let $\Delta=\Delta(H)$ and suppose $F$ is a graph with $\Delta(F)\leq 2\Delta-1$.
 It is sufficient to prove that $F\not\to H$.
 Consider a partition $V_1\dot\cup V_2$ of $V(F)$ with the maximum number of edges between $V_1$ and $V_2$.
 If there is a vertex $v\in V_1$ with at least $\Delta$ neighbors in $V_1$, then $v$ has at most $\Delta-1$ neighbors in $V_2$.
 Thus the partition $(V_1\setminus\{v\})\dot\cup (V_2\cup\{v\})$ has at least one more edge between the parts than the original partition, a contradiction.
 Hence both $F[V_1]$ and $F[V_2]$ have maximum degree at most $\Delta-1$.
 Color all edges between $V_1$ and $V_2$ red and all other edges blue.
 Then the red subgraph is bipartite and the blue subgraph has maximum degree at most $\Delta-1$.
 Thus $F\not\to H$.
 
 The lower bound is tight since $K_{2\Delta+1}\to H$, $\Delta\geq 3$, where $H$ is the graph of maximum degree $\Delta$ obtained from $K_{1,\Delta}$ by adding an edge between two leaves.
\end{proof}

%For any graph $G$ and any natural number $n$ the {\em extremal number ${\rm ex}(n,G)$} is defined to be the largest number $k$ such that there exists an $n$-vertex graph $F$ on $k$ edges that does not contain $G$ as a subgraph, i.e., ${\rm ex}(n,G) = \max\{ |E(F)| : G \not\subseteq F, |V(F)| = n\}$.

%\begin{lemma}\label{lem::extremalNo}
% Let $G$ be a graph and $H$ be a connected graph. If there is an integer $n$ such that ${\rm ex}(n, G) < \max \{ {\rm ex}(n,H)/2,   \sqrt{n}\,{\rm ex}(\sqrt{n},H)\} $, then $G \not\req H$.
% 
% In particular, if $G$ is a tree and $H$ contains a cycle, then $G \not\req H$.
%\end{lemma}
\begin{lemma}\label{lem::extremalNo}
Let $G$ and $H$ be graphs.\\
If ${\rm ex}(n, G) < {\rm ex}(n,H)/2$ or if $H$ is connected and ${\rm ex}(n, G) < \sqrt{n}\,{\rm ex}(\sqrt{n},H)$, then $G\not\req H$.

In particular, if $G$ is a forest and $H$ contains a cycle, then $G \not\req H$.
\end{lemma}
\begin{proof}
 Assume first that  ${\rm ex}(n,G)<{\rm ex}(n,H)/2$.  Let  $F$ be a graph  on $n$ vertices with ${\rm ex}(n,H) \geq 2\,{\rm ex}(n,G)+1$ edges without a copy of $H$.
 In any $2$-coloring of the edges of $F$ one of the color classes contains at least ${\rm ex}(n,G)+1 $ edges, and thus a copy of $G$.
 Hence $F\to G$, but $F\not\to H$.
	
 Assume now that $H$ is connected and ${\rm ex}(n^2, G) < n\,{\rm ex}(n,H)$.  Let $F$ be a graph on $n$ vertices and ${\rm ex}(n,H)$ edges not containing $H$.
 Let $F^\ast = F \times F$ be the Cartesian product of $F$ with itself, that is, $V(F^\ast) = V(F) \times V(F)$ and $\{(u,v),(x,y)\} \in E(F^\ast)$ if and only if $u = x$ and $vy \in E(F)$ or $v=y$ and $ux \in E(F)$.
 Then $F^\ast$ has $n^2$ vertices and $2n\,{\rm ex}(n,H)$ edges. 
 In any  $2$-edge-coloring of $F^\ast$ there is a color class with at least $n\,{\rm ex}(n,H)$ edges.
 This color class contains a copy of $G$, thus  $F^\ast\to G$.
 On the other hand, we can color the edges of  $E(F^\ast)$   without creating monochromatic copies of $H$ by  coloring  an edge $\{(u,v),(x,y)\}$ red if $u=x$ and blue otherwise.
 Note that each color class is a vertex disjoint union of $n$ copies of $F$ and thus does not contain $H$, as $H$ is connected.
 Thus  $F^\ast \not\to H$.
 \medskip
 
 For the second part of the statement let $G$ be any forest and $H$ be any graph with a cycle $C$.
 We have ${\rm ex}(n,G) \leq |V(G)|n$ and ${\rm ex}(n,H) \geq {\rm ex}(n,C) \geq \Omega(n^{1+\frac{1}{|V(C)|-1}})$~\cite{LargeGraphsHighGirth}.
 Hence for sufficiently large $n$ we have ${\rm ex}(n,G) < {\rm ex}(n,C)/2$ and thus $G\not \req H$ by the first part of the Lemma.
% there is a graph $F$ with $F \to G$ and $F \not\to C$.
% In particular $F \not\to H$ and thus $G \not\req H$.
\end{proof}

Let $Z_1$ and $Z_4$ denote the graphs obtained from $C_4$ by adding two, respectively three, pendent edges at some vertex,  let  $Z_5$ denote the graph obtained from $K_{2,3}$ by adding a pendent edge at a vertex of degree $3$, see Figure~\ref{fig:K23Special}.
	
	 \begin{figure}[htbp]
 \centering
 \includegraphics{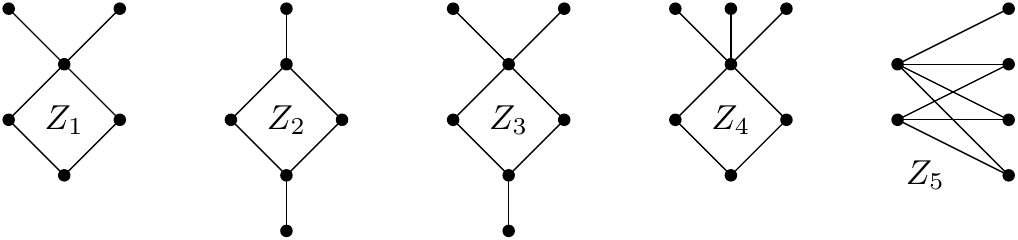}
 \caption{The graphs $Z_1$, $Z_2$, $Z_3$, $Z_4$ and $Z_5$.}
 \label{fig:K23Special}
\end{figure}

  The next technical lemmas we prove in the Appendix.  
  
   \edef\appendixLemmas{\thelemma} % saves lemma counter for using it later in appendix
  \begin{lemma}\label{lem:minDegG4}
  $\cR_{\delta}(Z_4)=1$.
  \end{lemma}

 \begin{lemma}\label{lem:NoMonoKab}
 In any $2$-edge-coloring of $K_{3,13}-e$ without monochromatic $Z_5$ the vertex of degree $2$ is incident to exactly one red and one blue edge.
\end{lemma}
 
\begin{lemma}\label{lem::minDegG5}
 $\cR_{\delta}(Z_5)=1$.
\end{lemma}

\begin{lemma}\label{lem:W4Isolated}
 If $H$ is connected and has at least $6$ vertices, then $H\not\req W_4$.
\end{lemma}
 
\begin{lemma}\label{lem:K23Isolated}
 If $H$ is connected and has at least $6$ vertices, then $H\not\req K_{2,3}$.
\end{lemma}

\section{Proofs of Theorems}\label{Proofs}
% In this section we proof the main theorems.

%%%%%%%%%%%%%%%%%%%%%%%%%%%%%%%%%%%%%%%%%%%
%%%%%%%%   Different Chromatic No.   %%%%%%%%%%%%%%
%%%%%%%%%%%%%%%%%%%%%%%%%%%%%%%%%%%%%%%%%%%

\subsection{Proof of Theorem~\ref{thm::DifferentChrNo}}

If $\chi(G)=2$ then $G$ is not Ramsey equivalent to any graph $H$ of higher chromatic number by Observation~\ref{obs::bipartite}.
% because $K_{n,n} \to G$ for sufficiently large $n$ and $K_{n,n} \not\to H$. 
So, assume that $\chi(G)\geq 3$.

 Let $\omega = \omega(G)$, $k=\chi(G)$, $\chi(H)>k$.   
 We assume $\omega(G)=\omega(H)$, otherwise $G\not\req H$ by Lemma~\ref{lem::RamseyGraphSameCliqueNo}. 
 Let $V(G) = V_1\cup V_2$, $V_1 \cap V_2 = \emptyset$,  such that $G_1=G[V_1]$ and $G_2=G[V_2]$ each  have clique number less than $\omega$.
Let $G^\star$ be a vertex disjoint union of $G_1$ and $G_2$, in particular $\omega(G^\star) < \omega$.
 We shall construct a graph $\Gamma$ such that $\Gamma\to G$ and $\Gamma \not\to H$.

 The building blocks of $\Gamma$ are  a hypergraph $\HH$ and graphs $F$ and $F'$ such that:
 \begin{itemize}
 \item{} $\HH$ is a $3$-chromatic, $k$-uniform hypergraph of girth at least $|V(H)|+1$.
  It exists by Lemma~\ref{lem::HypergraphGirthChromatic}. 
 \item{}  $F$ is  a graph such that $\omega(F) <\omega$  and   every set of at least $\epsilon_1|V(F)|$ vertices in $F$ contains a  copy of $G^\star$, where $\epsilon_1=2^{-|V(\HH)|}$.
  Such a graph exists by Lemma~\ref{lem::SparseGraphRichCopies}.
 \item{} $F'$ is a graph such that $\omega(F')=\omega(F)<\omega$ and  $F'\epsto F$ for $\epsilon = 2^{- |V(\HH)||V(F)|}$. Such a graph exists by    Lemma~\ref{lem::epsRamsey}.
 \end{itemize}
 Note that  $|V(\HH)|$  depends on  $|V(H)|$ and $\chi(G)$;     $|V(F)|$ in turn depends on $|V(\HH)|$,  $\omega(G)$,  and $G$, so $|V(F)|$ depends only on $H$ and $G$. So, $\epsilon$ and $\epsilon_1$ 
 are constants depending on $H$ and $G$.

%  \vskip 0.3 in
 
 Construct a   graph $\Gamma$ by replacing the vertices $v_1, \ldots, v_n$  of $\HH$ with  pairwise  vertex disjoint   copies  of $F'$ on vertex sets $V_1, \ldots, V_n$  and placing  a complete bipartite graph between two copies of $F'$ if and only if  the corresponding vertices belong to the same hyperedge of $\HH$, see Figure~\ref{fig:clique-splittable_1}.
 
 \begin{figure}[htb]
  \centering
  \includegraphics{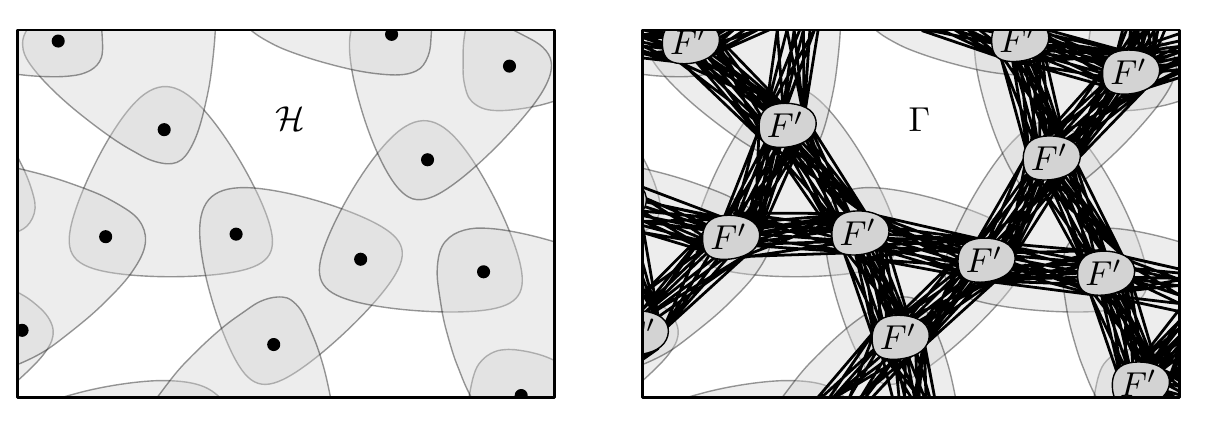}
  \caption{Left: The $k$-uniform hypergraph $\mathcal{H}$, $k=3$. Right: The graph $\Gamma$. }
  \label{fig:clique-splittable_1}
 \end{figure}

To show that  $\Gamma \not\to H$ color each edge with both endpoints in some $V_i$  red, $i=1, \ldots, n$,  and all other  edges  blue.  The red subgraph is a vertex disjoint union of copies of $F'$,  its clique number is strictly less than $\omega$, so it does not contain $H$, whose clique number is $\omega$. 
The blue subgraph is a union of complete $k$-partite graphs induced by $V_i$, $i=1, \ldots, n$. 
To see that the blue subgraph does not contain a copy of $H$, consider any copy of $H$ in $\Gamma$ and consider sets $V_{i_1}, \ldots, V_{i_\ell}$ intersecting the vertex set of this copy.
Since $\HH$ has girth at least $|V(H)|+1$,   $v_{i_1}, \ldots, v_{i_\ell}$ do not form a cycle in $\HH$, thus the blue graph induced by $V_{i_1}, \ldots, V_{i_\ell}$ is $k$-partite. 
However,  $\chi(H)>k$, so the blue subgraph does not contain a copy of $H$.\\

Next we shall show that  $\Gamma \to G$.  Consider a $2$-edge-coloring  of $\Gamma$. Recall that $n=|V(\HH)| = n(k, |V(H)|)$. 
We write $v_i\sim v_j$ if there is a hyperedge in $\HH$ containing both $v_i$ and $v_j$.
%  To simplify the upcoming arguments we consider a super graph  $\Gamma'$ of $\Gamma$ obtained by adding all missing edges between $V_i$ and $V_j$,  $1\leq i< j\leq n$, and $2$-color these edges 
% arbitrarily.

\setcounter{claim}{0}
\begin{claim}\label{claim1}
 For any $m$, $1\leq m \leq n$, any $i$,  $1 \leq i \leq m$, $V_i$ contains a subset $V_i'$ that is the vertex set of a monochromatic copy of $F$ and such that for any $v\in V_i'$ and any $j$ with $v_i\sim v_j$, $i < j \leq m$, all edges from $v$ to $V_j'$, are of the same color.
\end{claim}

\begin{figure}[htb]
 \centering
 \includegraphics{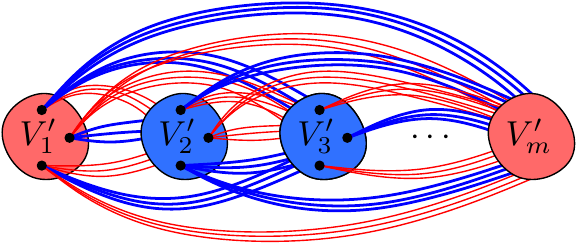}
 \hspace{4em}
 \raisebox{1.8em}{\includegraphics{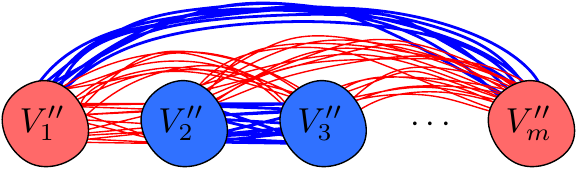}}
 \caption{Illustrations of Claim~\ref{claim1} (left) and Claim~\ref{claim2} (right). Thin edges are red and thick edges blue.}
 \label{fig::claims}
\end{figure}

We prove Claim~\ref{claim1} by induction on $m$ using the Focusing Lemma (Lemma~\ref{lem::focusing}).
When  $m=1$, we see that $\Gamma[V_1]$ is isomorphic to $F'$ and $F'\epsto F$. 
So in particular $F'\to F$ and there is a monochromatic copy of $F$ on some vertex set $V'_1$.
Assume that $V_1', V_2', \ldots, V_m'$ form vertex sets of monochromatic copies of $F$  satisfying the conditions of Claim~\ref{claim1}.
Apply the Focusing Lemma to the bipartite graph with parts $U_m= V_1'\cup \cdots \cup V_m'$ and  $V_{m+1}$.
It gives a subset $V_{m+1}^*\subseteq V_{m+1}$ such that 
for any $v\in U_m$, all edges between $v$  and $V_{m+1}^*$, if any, are of the same color and such that $|V_{m+1}^*|\geq 2^{-|V_1'\cup \cdots \cup V_m'|  } |V_{m+1}| = 2^{-m|V(F)|} |V_{m+1}| \geq  \epsilon |V_{m+1}|$.
Thus $\Gamma[V_{m+1}^*]$ contains a monochromatic copy of $F$, because $\Gamma[V_{m+1}]$ is isomorphic to $F'$ and  $F'\epsto F$.  Call the vertex set of this copy $V_{m+1}'$.

\begin{claim}\label{claim2}
 For any $m$, $1\leq m\leq n$, and $i$, $m \leq i \leq n$, each $V_i$ contains a subset $V_i''\subseteq V_i'$ that is the vertex set of a monochromatic copy of $G^\star$ and such that for each $j$ with $v_i\sim v_j$, $i < j \leq n$, $V_i''$, $V_j''$ are partite sets of a monochromatic complete bipartite graph.
\end{claim}

We prove Claim~\ref{claim2} by induction on $n-m$ using the pigeonhole principle. 
When  $m=n$, we see that $V_n'$ forms the vertex set of a monochromatic $F$, that in turn  contains  a monochromatic $G^\star$.  Denote the vertex set of this 
$G^\star$ as  $V''_n$.
Assume that $V_m'', V_{m+1}'', \ldots, V_n''$ form vertex sets of monochromatic copies of $G^\star$ satisfying the conditions of Claim~\ref{claim2}.
Consider $V_{m-1}'$ and recall from Claim~\ref{claim1} that each vertex in $V_{m-1}'$ sends only red or only blue edges to each $V_i''$ with $v_{m-1}\sim v_i$,  $i=m, \ldots, n$.
If $v_{m-1}\sim v_n$ then at least half of the vertices in $V_{m-1}'$ send monochromatic stars of the same color to $V_n''$.
If $v_{m-1}\sim v_{n-1}$ then at least half of those send monochromatic stars of the same color to $V_{n-1}''$, and so on.
So at least $2^{-(n-m)} |V_{m-1}'|$ vertices of $V_{m-1}'$ send monochromatic stars of the same color to each $V_i''$ with $v_{m-1}\sim v_i$ for $i=m, \ldots, n$. We denote the set of these vertices by $V_{m-1}^*$.
Since $\Gamma[V_{m-1}']$ forms the vertex set of a monochromatic $F$, and $|V_{m-1}^*|\geq 2^{-(n-m)}  |V_{m-1}'|\geq  \epsilon_1|V_{m-1}'| $,
the definition of $F$ implies that $\Gamma[V_{m-1}^*]$ contains a monochromatic copy of $G^\star$. We denote the vertex set of this copy by $V_{m-1}''$.

 \begin{figure}[htb]
  \centering
  \includegraphics{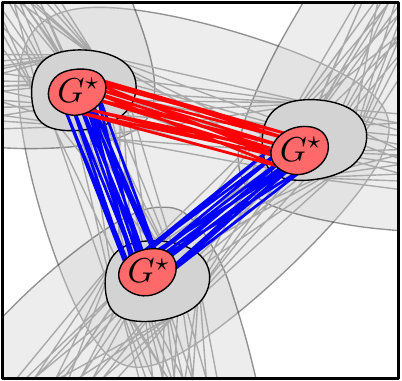}
  \caption{A set of $k$ red copies of $G^\star$ corresponding to the $k$ vertices of a hyperedge of $\mathcal{H}$. Here $k=3$. The complete bipartite graph between any two of these $k$ copies is also monochromatic.}
  \label{fig:clique-splittable_2}
 \end{figure}

Applying Claim~\ref{claim2} with $m=1$, we see that each vertex $v_i$ of $\HH$ corresponds to a monochromatic copy of $G^\star$ with vertex set $V_i''$, such that 
all edges between any two such copies from a common hyperedge have the same color.
Assigning the color of this $G^\star$ to $v_i$ gives a $2$-coloring of $V(\HH)$.
Since $\chi(\HH)>2$, there is a monochromatic hyperedge,  without loss of generality with red vertices $v_1, \ldots, v_k$. Thus in $\Gamma$ there are $k$ red copies of $G^\star$ on vertex sets $V_1'', \ldots, V_k''$, 
such that  $V_i'', V_j''$  are partite sets of monochromatic complete bipartite graphs, for all $i, j$,  $1\leq i<j\leq k$, see Figure~\ref{fig:clique-splittable_2}.
If at least one such bipartite graph is red, then there is a red copy of $G$ obtained by taking a red $G_1\subseteq G^*$ from one part and a red $G_2\subseteq G^*$ from the other part.
So we can assume that all such bipartite graphs are blue, forming a complete $k$-partite graph with each part of size $|V(G)|$.  Since $\chi(G)=k$,  there is a blue copy of $G$.  Thus $\Gamma\to G$.  Since  $\Gamma \not\to H$, we have that $G\not\req H$.
This concludes the proof of Theorem~\ref{thm::DifferentChrNo}.
\qed

 \begin{proof}[Proof of Corollary~\ref{Cor}]
Let $G$ and $H$ be two graphs such that $\chi(G)\neq \chi(H)$ and $\chi(G)\leq 2\omega(G)-2$.
Consider an arbitrary proper $\chi(G)$-vertex-coloring of $G$.
Let $V_1$ denote the union of $\lfloor\frac{\chi(G)}{2}\rfloor$ color classes and $V_2=V(G)\setminus V_1$.
Since $\omega(G)\geq \frac{\chi(G)}{2} +1$, every maximum clique contains a vertex from both sets $V_i$, $i=1,2$.
Thus, $G$ is clique splittable.
So, if $\chi(H)>\chi(G)$ then $G\not\req H$ by Theorem~\ref{thm::DifferentChrNo}.
If  $\chi(H)< \chi(G)$, then $\chi(H) <\chi(G) \leq 2\omega(H) -2$ (where we assume $\omega(H)=\omega(G)$ by Lemma~\ref{lem::RamseyGraphSameCliqueNo}). Thus, $H$ is clique-splittable with the same arguments as above.
Hence $G\not\req H$ by Theorem~\ref{thm::DifferentChrNo}.
\end{proof}

%%%%%%%%%%%%%%%%%%%%%%%%%%%%%%%%%%%%%%%%%%%
%%%%%%%%   Same Chromatic No.   %%%%%%%%%%%%%%
%%%%%%%%%%%%%%%%%%%%%%%%%%%%%%%%%%%%%%%%%%%

\subsection{Proof of Theorem~\ref{thm::SameChrNo}}

Our construction is similar to the one from Lemma 3.9 in~\cite{Grinshpun_RamsesyMinimal}.

 Consider a connected graph $H$.
 We may assume $\omega(G)=\omega(H)$ by Lemma~\ref{lem::RamseyGraphSameCliqueNo}.
 Note that $G$ is clique-splittable since $\omega(S)=0$ and $\omega(G-S)<\omega(G)$.
 Further note that if $G$ is bipartite, the conditions of the theorem imply that $G$ is a union of a matching and a set of  independent vertices.
 However, $G$ is assumed to be connected, and thus it must be a single edge.  Since a single edge is Ramsey isolated, we 
 can assume that $\chi(G)\geq 3$.\\

 In the first part of the proof, we assume that $H\not\subseteq G$ and $\chi(H)\geq \chi(G)$.
 Let $s=|S|$, $k= \chi(G) = \chi(H)$ and let $m$ denote the size of a matching induced by two color classes of some proper $k$-vertex-coloring of $G$.
 Note that $m\geq 1$ since there is at least one edge between any two color classes.
 Further let $n=|V(G)|$, $\omega=\omega(G)$ and $G_S$ be a vertex disjoint union of $G-S$ and $S$ independent vertices, 
 i.e., $G_S$ is the graph obtained from $G$ by deleting all edges incident to $S$.
 Then $\omega(G_S)<\omega$.
 Let $G'$ be a vertex disjoint union of $m' = (k -2)(s-1)+m$ copies of $G_S$ and $G_0'$ be a vertex disjoint union of $m'$ copies of $G$.
%Finally we define $m''=(k-2)(s-1) + m$ and 
 Let $\epsilon = 2^{-|V(G')|k} = 2^{-m'nk}$.
%  *** Adjust epsilon to a new setting. $2^{-|V(G')k}$ should be enough.*****\\
 Let $F$ be a graph with $F\epsto G'$ and $\omega(F) = \omega(G') < \omega$, which exists by Lemma~\ref{lem::epsRamsey}.
 We construct a graph $\Gamma$ by taking the vertex disjoint union of a copy of $G_0'$ and $k-2$ copies of $F$ denoted by $F_1,\ldots,F_{k-2}$ and placing a complete bipartite graph between $F_i$ and $F_j$,  $1\leq i < j \leq k-2$  and between $F_i$ and $G_0'$, $i=1, \ldots, k-2$, see Figure~\ref{fig:matching-coloring_1}.\\
  
 \begin{figure}[htb]
  \centering
  \includegraphics{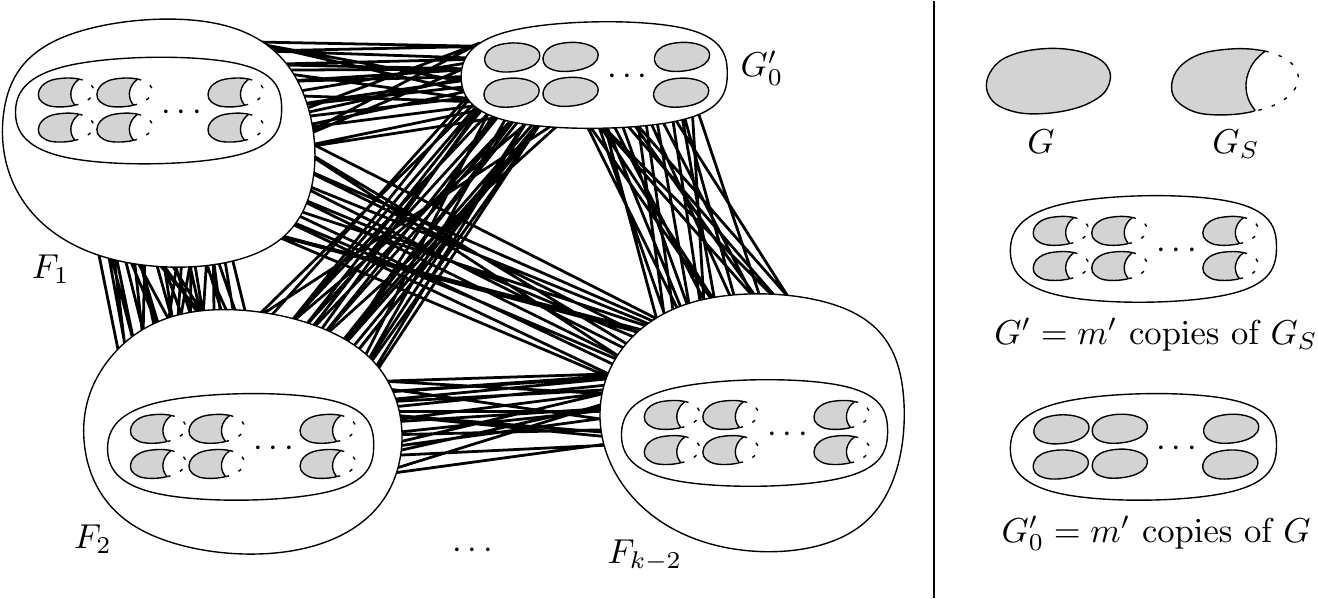}
  \caption{The graph $\Gamma$ consisting of $k-2$ copies $F_1,\ldots,F_{k-2}$ of $F$ and one copy of $G'_0$ and all possible edges between distinct copies. We have $F \epsto G'$, $G'$ consists of $m'$ copies of $G_S$ and $G'_0$ consists of $m'$ copies of $G$.}
  \label{fig:matching-coloring_1}
 \end{figure}

 We shall show that $\Gamma \to G$, but $\Gamma \not\to H$. 
 Color all edges within each $F_i$ and within $G_0'$ red and all other edges blue.
 Since $\omega(F)< \omega = \omega(H)$, $H\not\subseteq F$, and thus $H\not\subseteq F_i$, $i=1, \ldots, k-2$.
 Since $H\not\subseteq G$ and $H$ is connected,  we have that $H\not\subseteq G_0'$. 
 Thus there is no red copy of $H$.
 On the other hand, the blue subgraph is a complete $(k-1)$-partite graph,  but $\chi(H) \geq \chi(G) = k$. 
 Thus there is no blue copy of $H$.
 
 It remains to show that $\Gamma\to G$.
 Consider a $2$-edge-coloring of $\Gamma$.
 Assume for the sake of contradiction that there is no monochromatic copy of $G$.
 We prove the following claim, similar to Claim~\ref{claim-1} in the proof of Theorem~\ref{thm::DifferentChrNo}, by induction on $p$ (up to renaming colors), see Figure~\ref{fig:matching-coloring_2} for an illustration.

 \begin{claim*}
  For each $p$, $1\leq p\leq k-2$, and each $i$, $1\leq i\leq p$, there is a red copy $G_i'$ of $G'$ in $F_i$. Moreover for each $i$, $0 \leq i < p$, each vertex $v$ in $G_i'$ and each $j$, $i < j \leq p$, all edges between $v$ and $G_j'$ are of the same color.
 \end{claim*}
 
  \begin{figure}[htb]
  \centering
  \includegraphics{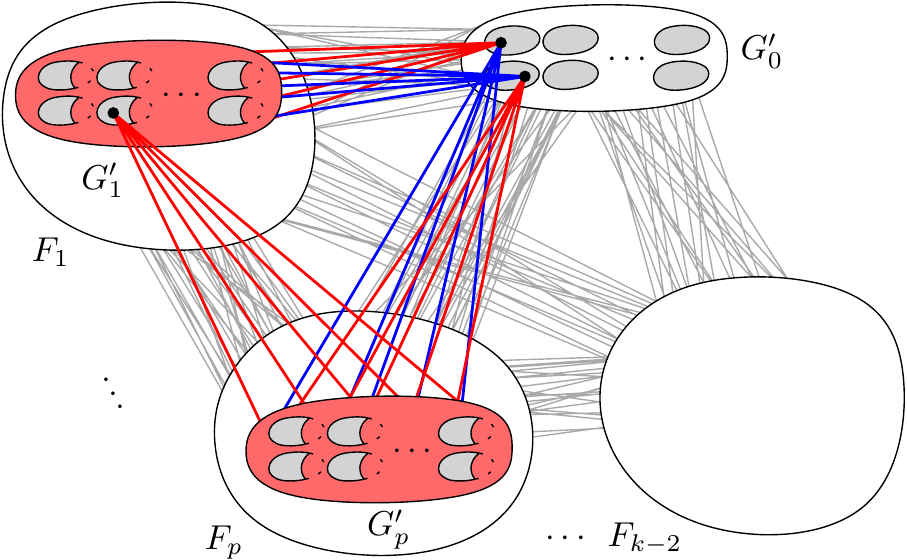}
  \caption{Illustrating the statement of the Claim.}
  \label{fig:matching-coloring_2}
 \end{figure}
 
 There is a set $V_1$ of $2^{-m'n}|V(F_1)| \geq \epsilon |V(F)|$ vertices in $F_1$ such that for each vertex in $G_0'$ all edges to $V_1$ are of the same color by the Focusing Lemma (Lemma~\ref{lem::focusing}).
 Since $F\epsto G'$ there is a monochromatic  copy $G'_1$ of $G'$ in $F_1[V_1]$.
 Assume without loss of generality that $G'_1$ is red.
 This proves the Claim for $p=1$, and for $k=3$. 
 
 Suppose $k-2\geq p \geq 2$ and there are red subgraphs $G'_{1},\ldots,G'_{p-1}$, satisfying the conditions of the Claim.
 We apply the Focusing Lemma to the complete bipartite graph with one part $V(G'_0)\cup \cdots \cup  V(G'_{p-1})$ and the other part $V(F_p)$.
 There is a set $V_p\subseteq V(F_p)$ of size $2^{-|V(G')|p} |V(F_p)|\geq \epsilon |V(F)|$, such that for each vertex $v$ in $G'_0, \ldots, G'_{p-1}$ all edges from $v$  to $V_p$ are of the same color.
 Since $F\epsto G'$ there is a monochromatic copy $G_p'$ of $G'$ in $F_p[V_p]$.
 It remains to prove that $G_p'$ is red. 
 Assume $G_p'$ is blue.
 Consider the vertices of $G_1'$.
 All of them send monochromatic stars to $G_p'$.
 At most $s-1$ of these stars are blue, as otherwise these stars together with a blue subgraph of $G_p'$ isomorphic to $G_S$ form a blue copy of $G$.
 Since the number of vertex disjoint copies of $G_S$ in $G_1'$ is $m'>s-1$, there is a red copy $G^*$ of $G_S$ in $G'_1$ whose vertices send only red stars to $G_p'$.
 Taking $G^*$ and $s$ vertices from $G_p'$ gives a red copy of $G$, a contradiction.
 So we may assume that $G_p'$ is red, which completes the proof of the Claim.\\

 Consider the red $G_i'$, $1\leq i\leq k-2$, given by the Claim for $p=k-2$.
 We say that a vertex in $V(G_i')$, $i=0, \ldots, k-3$  is bad for $G_j'$ if it sends a red star to $G_j'$, for some $j>i$.
 Since for each $G_j'$ there are at most $s-1$ bad vertices, there are at most $(k-2)(s-1)$ bad vertices overall.
%  Since each $G_i'$, $i=0, \ldots k-2$, has $m' = (k-2)(s-1) + m$ disjoint copies of $G_S$, deleting all bad vertices still leaves at least $m$ copies $G_1^0, \ldots, G_m^0$ of $G$ in $G_0'$ and at least one copy of $G_S$, denoted by $G_i''$, in each of $G_i'$, $i=1,\ldots, k-2$.
Since $G_0'$ has $m' = (k-2)(s-1) + m$ vertex disjoint copies of $G$, there are at least $m \geq 1$ copies $G_1^0, \ldots, G_m^0$ of $G$ in $G_0'$ without bad vertices.
Since each $G_i'$, $i=1, \ldots k-2$, has $m' = (k-2)(s-1) + m$ disjoint copies of $G_S$, there is at least one copy $G_i''$ of $G_S$ in $G_i'$ without bad vertices, $i=1,\ldots, k-2$.    
 Note that all $G_i''$s are red, $i=1, \ldots, k-2$,  all edges between them are blue, and all edges between a $G_i''$ and $G^0_j$ are blue, $i=1, \ldots, k-2$, $j= 1, \ldots, m$, see Figure~\ref{fig:matching-coloring_3}.
 
  \begin{figure}[htb]
  \centering
  \includegraphics{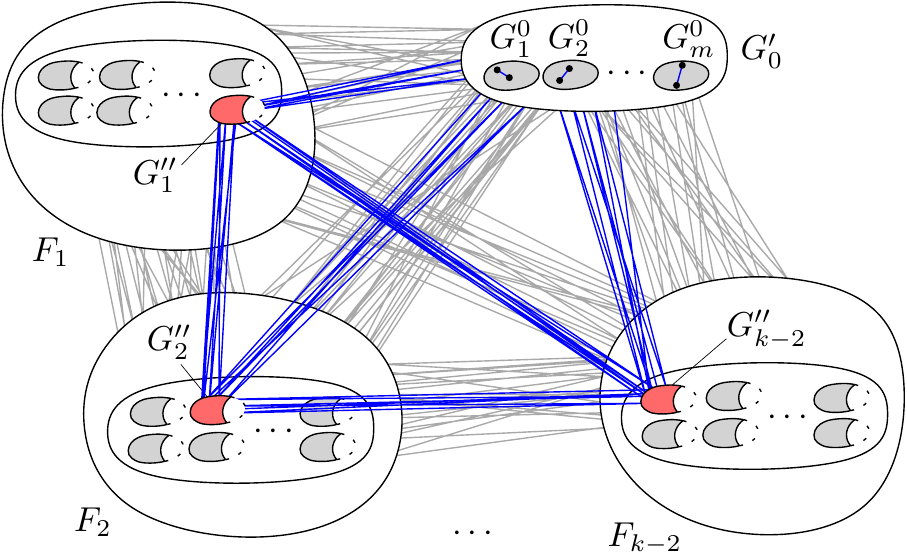}
  \caption{One red copy of $G_S$ in each of $F_1,\ldots,F_{k-2}$ and $m$ copies $G^0_1,\ldots,G^0_m$ of $G$ in $G'_0$ where all edges between distinct copies are blue. If each $G^0_i$, $i=1,\ldots,m$, has a blue edge we find a blue copy of $G$ in here.}
  \label{fig:matching-coloring_3}
 \end{figure}
 
 By assumption each $G^0_j$, $j=1, \ldots, m$,  has a blue edge, since otherwise there is a red copy of $G$.
 But then we can find a blue copy of $G$ by identifying these blue edges with the matching of size $m$ induced by the union of two color classes of $G$, picking the other vertices of these two color classes from $G^0_1$ and the vertices of the other $k-2$ color classes of $G$ from $G_i''$, $i=1, \ldots, k-2$.
 Since $|V(G_i'')| =|V(G)|$, there is sufficient number of vertices for each color class.
 Altogether we have a contradiction to our assumption that there are no monochromatic copies of $G$. Hence $\Gamma\to G$.
 This concludes the proof in case when $H\not\subseteq G$ and $\chi(H)\geq \chi(G)$.\\

 Now, in the second part of the proof, we assume that $H\subseteq G$.
Then $\chi(H)\leq \chi(G)$.
Since we assume that $\omega(G)=\omega(H)$,  we have $\omega(H-S) < \omega(H)$. Thus, $H$ is clique-splittable.
Assume first that $\chi(H)< \chi(G)$.
Then we have $G\not\req H$ by Theorem~\ref{thm::DifferentChrNo},  applied with roles of $G$ and $H$ switched.
The last case to consider is when  $\chi(H)=\chi(G)$ (and $H\subseteq G$).
%  Since $H\subseteq G$ we have $\omega(H)=\omega(G) \leq \chi(H)=\chi(G)$.
 Now any proper $\chi(G)$-vertex-coloring of $G$ with two color classes inducing a subgraph of a  matching gives such a coloring of $H$, too.
 Thus, the first part of the proof applied with roles of $G$ and $H$ switched shows that $G\not\req H$.
 \qed

% %%%%%%%%%%%%%%%%%%
% If $H$ is a subgraph of $G$ and $\chi(H) = \chi(G)$, then the two color classes inducing a matching in some proper coloring of $G$ induce a matching in $H$ as well.
% Moreover $\omega(H-S) < \omega(H)$ since we assume $\omega(G)=\omega(H)$.
% Altogether $H$ satisfies the assumptions of the theorem as well and w.l.o.g. we may switch roles of $G$ and $H$, such that $H$ is not a subgraph of $G$.
% Due to this fact we may assume that $H$ is not a subgraph of $G$.
% %%%%%%%%%%%%%%%
% 

%%%%%%%%%%%%%%%%%%%%%%%%%%%%%%%%%%%%%%%
%%%%%%%%    Proof of Path-Star thm %%%%%%%%%%%%%
%%%%%%%%%%%%%%%%%%%%%%%%%%%%%%%%%%%%%%%

\subsection{Proof of Theorem~\ref{thm::PathStar}}

\begin{proof}[Proof of \ref{thm::PathStar}.1:]

Assume that $\chi(G)=\omega(G)$ and in some proper $\chi(G)$-vertex-coloring of $G$ two color classes induce a subgraph of a  matching. Then  $G$ satisfies the requirements of Theorem~\ref{thm::SameChrNo}. If $\omega(H)\neq \omega(G)$ then $H\not\req G$ by Lemma~\ref{lem::RamseyGraphSameCliqueNo}. 
So, we can assume that $\omega(H)=\omega(G)$. If  $H \subseteq G$ or  $\chi(H)\geq \chi(G)$, then 
$G\not\req H$ by Theorem~\ref{thm::SameChrNo}. If  $H\not\subseteq G$ and $\chi(H)<\chi(G)$,
then $\omega(H)=\omega(G)  = \chi(G) >  \chi(H)$. Thus $\chi(H)<\omega(H)$, a contradiction.
\end{proof}

\begin{proof}[Proof of \ref{thm::PathStar}.2:]

	To see that a star $S=K_{1,t}$ is not Ramsey equivalent to any other graph, observe that $K_{1,2t-1}$ is a minimal Ramsey graph for $S$, but $K_{1,2t-1}$ is minimal Ramsey for neither any connected subgraph of $S$ nor  any connected graph that is not a subgraph of $S$.

 It remains to show that a path is not Ramsey equivalent to any other connected graph.
 Let $G = P_m$, a path on $m$ vertices, and $H$ be a connected graph not isomorphic to $G$.
 If $H$ is a path of different length, then $G\not\req H$ since $R(P_m) = m + \lfloor \frac{m}{2} \rfloor - 1$~\cite{PathRamsey} and hence $R(G)\neq R(H)$.
 So assume $H$ is not a path.
 If $H$ is not a tree, then by Lemma~\ref{lem::extremalNo} we have $G\not\req H$.  
 Otherwise, $H$ is a tree and $\Delta(H) \geq 3$. Then  $R_\Delta(H) \geq 2\Delta(H)-1 \geq 5$~\cite{DegreeRamseyTrees}, while an easy argument due to Alon \emph{\textit{et al.}}~\cite{SmallMonoComponents} shows that $R_\Delta(G) \leq 4$.
 Indeed, for any $4$-regular graph $F$ with girth at least $m+1$ we have $F \to P_m$ as follows.
 Considering any $2$-edge-coloring of $F$,  we see that since $F$ has average degree $4$ at least one color class has average degree at least $2$, i.e., contains a cycle.  Since $\girth(F) \geq m+1$, this monochromatic cycle has length at least $m+1$, and thus contains $P_m$. 
\end{proof}
 
 \begin{proof}[Proof of \ref{thm::PathStar}.3:]
\begin{figure}[htbp]
\centering
 \includegraphics{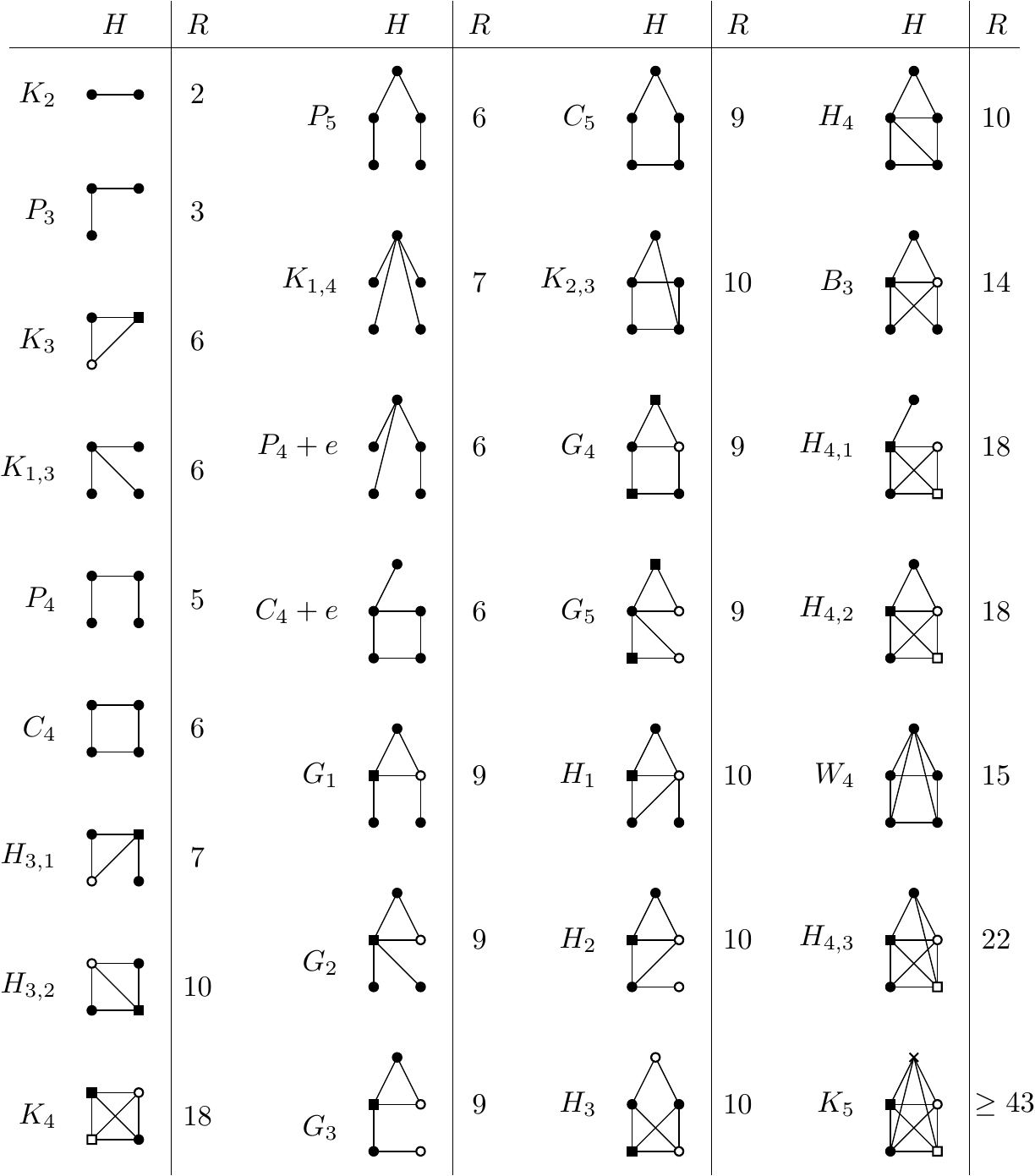}
 \caption{The connected graphs on at most $5$ vertices with their Ramsey numbers $R=R(H)$.}
 \label{fig::5Vert}
\end{figure}
 
 Figure~\ref{fig::5Vert} shows all non-trivial connected graphs on at most $5$ vertices.
 Let $S=\{C_4, P_4+e, C_4+e, C_5, K_{2,3}, H_4, W_4\}$.
 Observe that any connected graph on at most $5$ vertices which is not in $S$ satisfies the conditions of Theorem~\ref{thm::PathStar}.1 or~\ref{thm::PathStar}.2 and thus is Ramsey isolated; Figure~\ref{fig::5Vert} also indicates proper colorings for the graphs which satisfy the conditions of Theorem~\ref{thm::PathStar}.1.
 It remains to prove that each graph in $S$ is Ramsey isolated.
 We consider the graphs in $S$ grouped according to their Ramsey number.
Let $S_1=\{C_4, P_4+e, C_4+e\}$, $S_2=\{C_5\}$, $S_3=\{K_{2,3}, H_4\}$, and $S_4=\{W_4\}$.
Here $S_1$ contains the graphs from $S$ of Ramsey number $6$, $S_2$ the graph of Ramsey number $9$, $S_3$ those of Ramsey number $10$, and $S_4$ the graph of Ramsey number $18$, see~\cite{chvatal_Ramsey4Vertices,hendry_Ramsey5Vertices}.
 \medskip
 
 First of all we consider $G\in S_1$.
 Consider a connected graph $H$ which is not isomorphic to $G$.
 If $|V(H)|\geq 6$, then $R(H)>6$ and hence $H\not\req G$.
 Indeed, if $H$ is a star then coloring the edges of a $C_6$ in $K_6$ red and all other edges blue does not yield a monochromatic $H$.
 If $H$ is not a star, then color a copy of $K_{1,5}$ in $K_6$ red and all other edges blue.
 Then the red edges form a star and the blue connected subgraph contains only $5$ vertices, so the coloring has no monochromatic $H$.
 So assume that $|V(H)|\leq 5$.
 If $H\not\in S_1$ we have $G\not\req H$.
 Indeed either $H\in S\setminus S_1$ and $R(H)\neq R(G)$, or $H\not\in S$ and $H$ is Ramsey isolated by Theorem~\ref{thm::PathStar}.1 or~\ref{thm::PathStar}.2.
 So it remains to distinguish the graphs in $S_1$ from each other.
 We have $H_{5,4}\not\to C_4$ and $H_{5,4}\not\to C_4+e$ due to the coloring given in Figure~\ref{fig::H5x}.
We claim that $H_{5,4}\to P_4+e$.
Indeed, consider a $2$-edge-coloring of $H_{5,4}$ and a vertex $u$ of degree $5$.
Without loss of generality $u$ is incident to $3$ red edges $ux$, $uy$ and $uz$.
Then there is a red $P_4+e$ or all edges between $\{x,y,z\}$ and $V(H_{5,4})\setminus\{u,x,y,z\}$ are blue.
But in the latter case the vertices in $H_{5,4}$ other than $u$ give a blue $P_4+e$.
In particular $H_{5,4}\to P_4+e$ and thus $P_4+e\not\req C_4$ and $P_4+e\not\req C_4+e$.
Finally $\cR_{\delta}(C_4)=3$~\cite{Fox_MinDegreeRamseyMinimal} and $\cR_{\delta}(C_4+e)=1$~\cite{Fox_EquivClique} (for the latter see a remark in the conclusion of~\cite{Fox_EquivClique}).
Thus $C_4\not\req C_4+e$ and hence $G$ is Ramsey isolated.
\begin{figure}[tbp]
 \begin{minipage}{0.32\textwidth}
 \centering
 (1)
 \includegraphics[height=1.8cm]{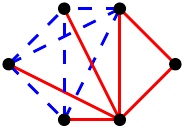}
 \end{minipage}
 \hfill
 \begin{minipage}{0.32\textwidth}
 \centering
 (2)
 \includegraphics[height=1.8cm]{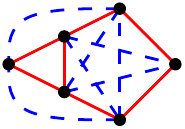}
 \end{minipage}
 \hfill
 \begin{minipage}{0.32\textwidth}
 \centering
 (3)
 \includegraphics[height=1.8cm]{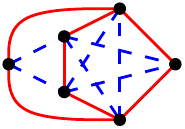}
 \end{minipage}
 \caption{A coloring of  $H_{5,2}$ without monochromatic $P_5$ (1), a coloring of $H_{5,4}$ without monochromatic $C_4$ (2) and a coloring of $H_{5,4}$ without monochromatic $K_3$ (3).}
 \label{fig::H5x}
\end{figure}
\medskip 

 Next consider $G\in S_2$, i.e., $G=C_5$ and a connected graph $H$ which is not isomorphic to $G$.
 If $|V(H)|\leq 5$ we have $G\not\req H$, because $H\in S\setminus S_2$ and hence $R(H)\neq R(G)$ or $H\not\in S$ (and $H$ is Ramsey isolated by Theorem~\ref{thm::PathStar}.1 or~\ref{thm::PathStar}.2).
 If $H$ is bipartite then $G\not\req H$ by Observation~\ref{obs::bipartite}.
 If $|V(H)|\geq 6$ and $H$ is not bipartite, then $R(H)> 10$.
 Indeed color the edges of $K_{10}$ with two vertex disjoint red copies of $K_5$ and all other edges blue.
 Then each connected component of the red subgraph has $5$ vertices and the blue subgraph is bipartite.
 In particular there is no monochromatic copy of $H$.
 We conclude that $G\not\req H$, so $G$ is Ramsey isolated.
 \medskip
 
 Next consider $G\in S_3$ and a connected graph $H$ which is not isomorphic to $G$.
 Since $K_{2,3}$ is bipartite but $H_4$ is not, the two graphs in $S_3$ are not Ramsey equivalent by Observation~\ref{obs::bipartite}.
 If $|V(H)|\leq 5$ then $G\not\req H$, because either $H\in S_3\setminus\{G\}$, or $R(H)\neq R(G)$, or $H\not\in S$.
 So assume $|V(H)|\geq 6$.
 Then $K_{2,3}\not\req H$ by Lemma~\ref{lem:K23Isolated}.
 If $H$ is bipartite then $H_4\not\req H$ by Observation~\ref{obs::bipartite}.
 If $H$ is not bipartite then $H_4\not\req H$, since $R(H)> 10=R(H_4)$ as argued above (when considering $S_2$).
 Altogether $G$ is Ramsey isolated.
 \medskip
 
 Finally consider $G\in S_4$, i.e. $G=W_4$, and a connected graph $H$ which is not isomorphic to $G$.
 If $|V(H)|\leq 5$ we have $G\not\req H$, because $H\not\in S_4$ and hence $R(H)\neq R(G)$ or $H\not\in S$.
 If $|V(H)|\geq 6$ then $W_4\not\req H$ by Lemma~\ref{lem:W4Isolated}.
 Hence $G$ is Ramsey isolated.
\end{proof}

\vskip 1cm

\begin{proof}[Proof of Remark~\ref{smallDistinguishingGraphs}.]
Next we show that all but $11$ pairs from the $\binom{31}{2}=465$ pairs of distinct connected graphs on at most $5$ vertices are distinguished by a  small graph.  For $447$ such pairs of such graphs $\{G,H\}$ we give a distinguishing graph on $\min\{R(G),R(H)\}$ vertices, which is clearly best-possible.
Among others, we will use graphs $\Gamma$ and $\Gamma'$ given in Figure~\ref{fig::graphGamma} and Figure~\ref{fig::GammaPrimeG3} respectively.
The graph $\Gamma'$ is obtained from $K_7$ by adding two independent vertices of degree $5$ such that these two vertices have exactly $4$ common neighbors.

\begin{figure}[htbp]
\begin{minipage}{0.48\textwidth}
 \centering
 \includegraphics{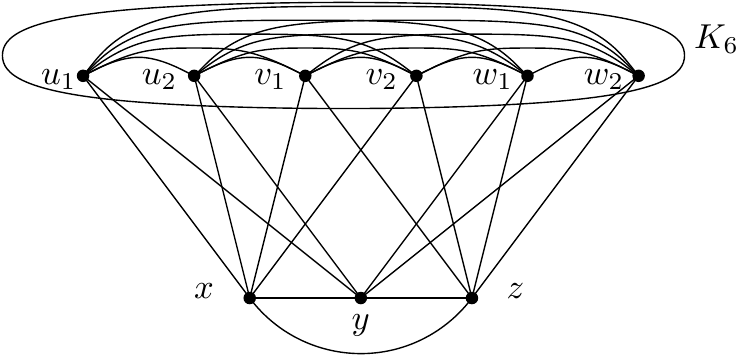}
 \caption{Graph $\Gamma$.}
 \label{fig::graphGamma}
\end{minipage}
\hfill
\begin{minipage}{0.48\textwidth}
 \centering
 \includegraphics{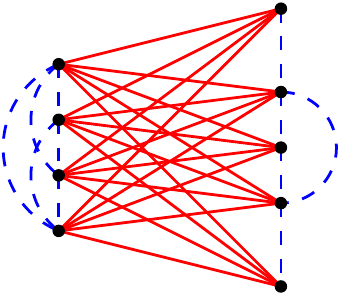}
 \caption{A coloring of $\Gamma'$ without monochromatic $G_3$.}
 \label{fig::GammaPrimeG3}
\end{minipage} 
\end{figure}

First of all note that two graphs $G$, $H$ of different Ramsey number are distinguished by $K_n$ where $n=\min\{R(G),R(H)\}$ which is the smallest possible order of a distinguishing graph.
This result distinguishes already lots of graphs using small graphs.
It remains to distinguish pairs of connected graphs on at most $5$ vertices of the same Ramsey number.
Hence we need to consider the following sets of graphs corresponding to Ramsey number $6$, $9$, $10$ and $18$ respectively;

\textbf{Ramsey number 6:} $\{K_3, K_{1,3},C_4, P_5, P_4+e, C_4+e\}$. We have
 $K_{1,5}\to K_{1,3}$ (pigeonhole principle) but $K_{1,5}\not\to K_3, C_4, P_5, P_4+e, C_4+e$ ($K_{1,5}$ does not contain these),
 $H_{5,2}\to P_4+e$ (Lemma~\ref{lem::H52P4e}) but $H_{5,2}\not\to K_3, C_4, P_5, C_4+e$ (Figure~\ref{fig::H5x}), 
 $H_{5,4}\to P_5$ (Lemma~\ref{lem::H53P5}) but $H_{5,4}\not\to K_3, C_4, C_4+e$ (Figure~\ref{fig::H5x}), 
 $K_{5,5}\to C_4, C_4+e$ (Lemma~\ref{lem::K55C4}) but $K_{5,5}\not\to K_3$ (since $K_3$ not bipartite). 
 It remains open to distinguish $C_4$ and $C_4+e$ by some small graph.

\textbf{Ramsey number 9:} $\{G_1, G_2, G_3, C_5, G_4, G_5\}$. We have 
 $\Gamma \to G_1, G_3$ (Lemma~\ref{lem::gammaG1},~\ref{lem::gammaG3}) but $\Gamma \not\to G_2, G_5$ (Lemma~\ref{lem::maxDegree}), 
 $\Gamma'\to G_1$ (Lemma~\ref{lem::GammaPrimeG1}) but $\Gamma'\not\to G_3$ (Figure~\ref{fig::GammaPrimeG3}), 
 $H_{8,5}\to G_1,G_2,G_3$ (Lemma~\ref{lem::H85G123}) but $H_{8,5}\not\to G_4, G_5, C_5$ (Figures~\ref{fig::H85},~\ref{fig::H86}).  
 We conjecture $H_{8,6}\to G_4$ (motivated by Lemma~\ref{lem::H86G4}) but $H_{8,6}\not\to G_5$ (Figure~\ref{fig::H86}). 
  It remains open to distinguish $C_5$ from $G_4$ and $G_5$ by small graphs.

  \begin{figure}[htbp]
 \begin{minipage}{0.48\textwidth}
 \centering
 \includegraphics{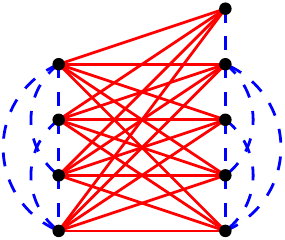}
 \caption{A coloring of $H_{8,5}$ without monochromatic $G_4$ and $C_5$.}
 \label{fig::H85}
 \end{minipage}
 \hfill
 \begin{minipage}{0.48\textwidth}
 \centering
 \includegraphics{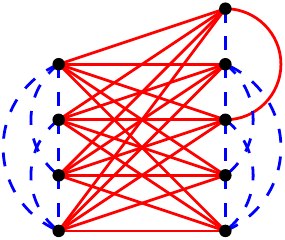}
 \caption{A coloring of $H_{8,6}$ without monochromatic $G_5$.}
 \label{fig::H86}
 \end{minipage}
\end{figure}
 
\textbf{Ramsey number 10:} $\{H_{3,2}, K_{2,3}, H_1, H_2, H_3, H_4\}$. We have 
 $K_{12,12}\to K_{2,3}$ (Lemma~\ref{lem::K1212}) but the other graphs are not bipartite, 
 $H_{9,6}\to H_{3,2}, H_1, H_2$ (Lemma~\ref{lem::H96H12}) but $H_{9,6}\not\to H_3,H_4$ (Figures~\ref{fig::H96H3},~\ref{fig::H98H4}). 
 In this case it remains to distinguish each pair within the sets $\{H_{3,2}, H_1, H_2\}$ and $\{H_3,H_4\}$ with a small graph.
 
 \textbf{Ramsey number 18:} $K_4, H_{4,1}, H_{4,2}$. 
 We have not found any small distinguishing graph for the pairs in this case.
 \begin{figure}[htbp]
	\begin{minipage}{0.48\textwidth}
	\centering
	\includegraphics[]{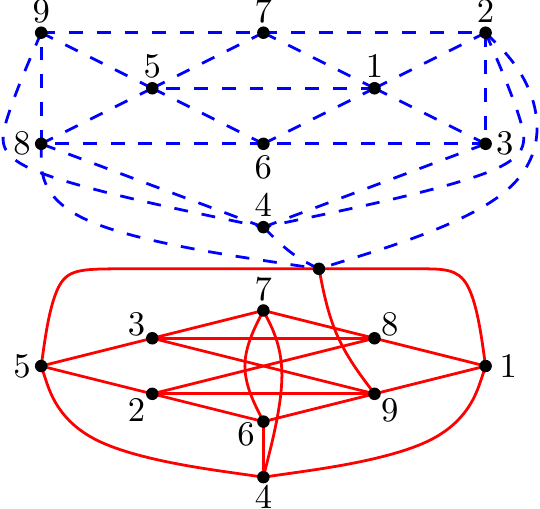}
	\caption{An edge-coloring of $H_{9,6}$ without monochromatic $H_3$ is obtained by identifying vertices of the same label.}
	\label{fig::H96H3}
	\end{minipage}
	\hfill
	\begin{minipage}{0.48\textwidth}
	\centering
	\includegraphics[]{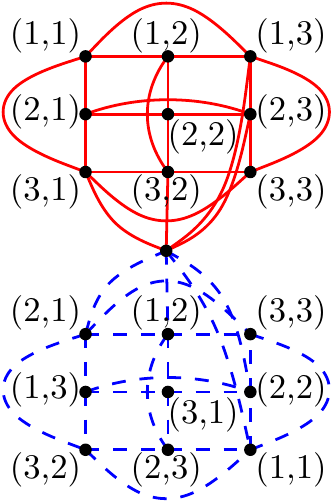}
	\caption{An edge-coloring of $H_{9,8}$ without monochromatic $H_4$ is obtained by identifying vertices of the same label.}
	\label{fig::H98H4}
	\end{minipage}
	\hfill
\end{figure}
\end{proof}

%%%%%%%%%%%%%%%%%%%%%%%%%%%%%%%%%%%%%%%
%%%%%%%%    Trees   %%%%%%%%%%%%%
%%%%%%%%%%%%%%%%%%%%%%%%%%%%%%%%%%%%%%%

\subsection{Proof of Theorem~\ref{thm::Trees} (Trees)}

Assume first that Conjecture~\ref{tree-conjecture} is true.  
Let $T_k$ and $T_\ell$ be trees on $k$ and $\ell$ vertices respectively, $k< \ell$. 
Note that ${\rm ex}(n, T_\ell) \geq  \frac{\ell - 2}{2} n - \ell^2$. Indeed, just take $\lfloor\frac{n}{\ell-1}\rfloor$ vertex disjoint copies of $K_{\ell-1}$. 
Then 
\[
 {\rm ex}(n, T_k) \leq  \frac{k-1 -\epsilon}{2} n  =\sqrt{n}  \left( \frac{k-1 -\epsilon}{2} \sqrt{n}\right) <   \sqrt{n} \left(  \frac{\ell- 2}{2}  \sqrt{n}  - \ell^2\right)  \leq  \sqrt{n}\hskip 0.1 cm {\rm ex}(\sqrt {n}, T_\ell),
\]
for sufficiently large $n$.
Thus ${\rm ex}(n, T_k) <  {\rm ex}(\sqrt {n}, T_\ell) \sqrt{n}$ and Lemma~\ref{lem::extremalNo} implies that $T_k \not\req T_{\ell}$.
\medskip

Now, we shall prove the second statement of Theorem~\ref{thm::Trees} without assuming the validity of Conjecture~\ref{tree-conjecture}.
Let  $T_k$ be  a balanced tree on $k$ vertices and $T_{\ell}$ be any tree on $\ell \geq k+1$ vertices.
Let $G$ be a $k$-regular graph of girth at least $k$, which is known to exist~\cite{SachsRegularGirth}.
We construct a bipartite $k$-regular graph $B$ of girth at least $k$ from $G$ by taking for each $v$ in $G$ two vertices $v_1,v_2$ in $B$ and for every edge $uv$ in $G$ the edges $u_1v_2$ and $u_2v_1$ in $B$.
Finally, let $F = L(B)$ be the line graph of $B$.
We shall show that $F \not\to T_{\ell}$ and $F\to T_k$.
 
 As $B$ is bipartite,   $F = L(B)$ is a union of two graphs $F_1$, $F_2$, each is a vertex disjoint union of copies of $K_k$, where each  clique in $F_i$ corresponds to 
 a set of edges incident to a vertex in the $i$th  partite set of $B$, $i=1,2$. Note that a clique in $F_1$ intersects a clique in $F_2$ by at most one vertex and 
 that each vertex in $F$ belongs to two cliques, one from $F_1$ and one from $F_2$.
 
 Coloring  $F_1$ red and $F_2$ blue  gives no  monochromatic $T_\ell$ since each monochromatic connected component has $k<\ell$ vertices.  
 Thus $F\not\to T_\ell$. 
  \medskip
 
 Next, we show that $F \to T_k$.
 Let $vw$ be an edge of $T_k$ such that the components of $T_k - vw$ rooted at $v$ and $w$ have order at most  $\lceil\frac{k+1}{2}\rceil$.
 Consider any edge-coloring of $F$ with colors red and blue.
 Note that $|V(F)| = |E(B)| = \frac{k}{2}|V(B)|$ and $|E(F)| = \binom{k}{2}|V(B)| = (k-1)|V(F)|$.
 Hence there are at least $\frac{k-1}{2}|V(F)|$ red edges or at least $\frac{k-1}{2}|V(F)|$ blue edges.
 (Note that Conjecture~\ref{tree-conjecture}, if true, would imply that there is a red or blue copy of $T_k$, independent of the girth of $B$ and whether $T_k$ is balanced.)
 Assume without loss of generality that there are at least $\frac{k-1}{2}|V(F)|$ red edges.
 Consider the red subgraph $G_r$ of $F$ and a subgraph $G$ of $G_r$ of highest average degree.
 It follows that $\delta(G) \geq \lceil\frac{k-1}{2}\rceil$ and $|E(G)| \geq \frac{k-1}{2}|V(G)|$, and so $\Delta(G) \geq k-1$.
 If $\Delta(G) = k-1$, then $G$ is $(k-1)$-regular and we can embed $T_k$ into $G$ greedily.
 So without loss of generality we have $\Delta(G) \geq k$.
  
 Let $x$ be a vertex of maximum degree in $G$, i.e., $\deg_G(x) \geq k$.
 It follows that $x$ has incident red edges in both corresponding maximum cliques $C_1,C_2$ in $F$.
 Without loss of generality $x$ has at least $\lceil\frac{k-1}{2}\rceil$ incident red edges in $C_1$.
 We embed $v$ onto $x$, $w$ onto a neighbor of $x$ in $G_r$ in $C_2$ and all neighbors of $v$ different from $w$ onto neighbors of $x$ in $G_r$ in $C_1$.
 Now we can greedily embed the subtrees $T_1,\ldots,T_a$ of $T_k - v$ with their roots at the designated vertices in $G_r$.
 Say $T_1$ is the subtree rooted at $w$.
 As $\delta(G) \geq \lceil\frac{k-1}{2}\rceil \geq |V(T_1)|-1,\sum_{i = 2}^a |V(T_i)|$ and $B$ has girth greater than $k$,  the embeddings of $T_1$ and $\bigcup_{i = 2}^a T_i$ are in disjoint sets of cliques.
 It follows that $F \to T_k$.
\qed

\subsection{Proof of Theorem~\ref{thm::MoreColors} (Multicolor Ramsey numbers)}

We prove the first part of the theorem.
Let $m=R(G,G,F)$.
Consider a $3$-edge-coloring $c$ of $K_m$ without red or blue $H$ and without green $F$, which exists as $m < R(H,H,F)$.
Let $\Gamma$ denote the graph obtained from $K_m$ by removing all green edges under $c$.
Thus $\Gamma\not\to H$ due to the coloring $c$ restricted to $\Gamma$.
But $\Gamma\to G$, since any $2$-edge-coloring of $\Gamma$ without monochromatic $G$ can be extended by the green edges of $c$ to an edge-coloring of $K_m$ without red or blue $G$ and without green $F$.\\

We prove the second statement by induction on $k$ with $k=2$ being obvious. 

Let $\Gamma$ be a graph such that $\Gamma\to_k G$, but $\Gamma\not\to_k H$, $k\geq 3$.
% Assume that $G\req H$.
Let $c$ be a $k$-edge-coloring of $\Gamma$ with no monochromatic $H$. 
Let a graph $\Gamma'$ be obtained from $\Gamma$ by deleting the edges of color $1$.
We have that  $\Gamma'\not\to_{k-1} H$ since $c$ restricted to $\Gamma'$ is a $(k-1)$-coloring with no monochromatic $H$. 
We claim that $\Gamma'\to_{k-1} G$, which, if true, gives $G\not\req_{k-1} H$ and by induction $G\not\req H$, as desired.

Let us assume for the sake of contradiction that $\Gamma'\not\to_{k-1} G$, i.e., there is a $(k-1)$-edge-coloring $c'$ of $\Gamma'$ without monochromatic $G$.
We see that there is a copy of $G$ in color $1$ of $c$, otherwise the coloring $c''$ of $\Gamma$ that is the same as $c'$ on $\Gamma'$ and that colors all other edges with color $1$ has no monochromatic $G$, a contradiction to the fact that $\Gamma\to_k G$.
Repeating the argument above to all colors in $c$, we see that each of them contains $G$.
More generally, we see that any edge-coloring of $\Gamma$ with $k$ colors avoiding monochromatic $H$ must have a monochromatic  $G$ in each color.
However, since $G\subseteq H$, $\Gamma'$ has no monochromatic $H$ under $c'$, and
%  Since  $\Gamma'\not\to_{k-1} G$ there is a $(k-1)$-coloring $c'$ of $\Gamma'$  with no monochromatic $G$. 
% In particular 
hence the coloring $c''$ of $\Gamma$ has no monochromatic $H$.
 Thus $c''$ must have monochromatic $G$ in each color, however there is no monochromatic $G$ in any of 
 the colors $2, \ldots, k$, a contradiction. \qed

%%%%%%%%%%%%%%%%%%%%%%%%%%%%%%%%%%%%%%%
%%%%%%%%    Conclusions   %%%%%%%%%%%%%
%%%%%%%%%%%%%%%%%%%%%%%%%%%%%%%%%%%%%%%

\section{Conclusions}\label{Conclusions}

This paper addresses Ramsey equivalence of graphs and gives a negative  answer to the question of Fox \textit{et al.}~\cite{Fox_EquivClique}:
``Are there two connected non-isomorphic graphs that are  Ramsey equivalent?'' for 
  wide families of graphs determined by so-called ``clique splitting'' properties and chromatic number. 
  In particular, we find an infinite family of graphs that are not Ramsey equivalent to any other connected graphs. 
This extends the only such  known family consisting of all cliques, paths, and stars.
\medskip
  
  Replacing  $\omega$   with any other ``nice'' Ramsey parameter,   $s$,  generalizes Theorems~\ref{thm::DifferentChrNo} and~\ref{thm::SameChrNo}. 
  Here, we say that a parameter $s$ is a ``nice'' Ramsey parameter, if for any graph $H$ and 
  any Ramsey graph $\Gamma$ for $H$, we have $s(\Gamma)\geq s(H)$ 
%   (    $s(\Gamma)\leq s(H)$ )
  and for all $\epsilon>0$ equality is attained for at least one $\Gamma$ with $\Gamma\epsto H$.
  So, both $\omega$ and $-g_0$ (the negative of the odd girth) are nice Ramsey parameters.
 \medskip
 
There are many questions that remain open in this area. Even the following weaker question is very far from being understood:
``Are there other graph parameters that distinguish graphs in a Ramsey sense?'',  i.e., is there a parameter $s$ such that 
$s(G)\neq s(H)$ implies that $R(G)\not\req R(H)$?  Here, we showed that the chromatic number, $\chi$, is very likely to be 
such a distinguishing parameter by proving this implication for graphs satisfying some additional properties. 
Interestingly enough, it is not clear, but most likely not true that $\chi(G)\neq \chi(H)$ implies that $\cR_\chi(G) \neq \cR_\chi(H)$. 
Indeed,  $\cR_\chi(K_4) = R(K_4) = 18$,   but  the positive answer to the Burr-Erd\H{o}s-Lov\'asz-Conjecture shows that  there is a $5$-chromatic graph $G$ with $R_\chi(G) = 4^2 + 1 = 17$, 
so $\chi(K_4)<\chi(G)$ but $\cR_\chi(K_4)> \cR_\chi(G)$.  We believe that there are infinitely many pairs of graphs $G, H$ of different chromatic number and the same value for $\cR_\chi$.
\medskip

In this paper, we addressed the relation between other types of Ramsey numbers and Ramsey equivalence and got results in terms of multicolor Ramsey numbers.
The following questions are open.  If $R(G,F)\neq R(H,F)$ for some graph $F$, does this imply $G\not\req H$?  For any two  non-isomorphic graphs $G,H$,  is there an integer $k$  such that $R_k(G)\neq R_k(H)$? Here $R_k(G)$ is the smallest integer $n$ such that any coloring of edges of $K_n$ with $k$ colors 
contains a monochromatic copy of $G$.
For example,  we see that $R( P_4+e)=R(K_{1,3})=  6$, and    $R_k(P_4+e)>2k+2=R_k(K_{1,3})$, for odd $k> 3$.
Another question is whether  the fact that  $R_k(G)\neq R_k(H)$ for some $k$ implies that  $G\not\req H$.  We answered the last question in positive only when $G$ is a subgraph of $H$.
\medskip
 
Cliques play a special role in Ramsey theory and got a particular attention in Ramsey equivalence. Still, it is not clear for what graphs is a  clique a  minimal Ramsey graph. 
Specifically,  if the size Ramsey number $R_e(H)$ is less than  $\binom{R(H)}{2}$, does it imply that $K_{R(H)}$ is not a minimal Ramsey graph for $H$?
\medskip

A positive answer to the following question would immediately give a negative answer to the question of Fox \textit{et al.}: ``Is there a graph in the Ramsey class of any connected graph $G$, that does not belong to the Ramsey class of any other connected graph, except for subgraphs of $G$?''
\medskip

It is also not clear how small could be a distinguishing graph for two not Ramsey equivalent graphs. 
Is there a function $f$ such that for any two not Ramsey equivalent graphs $G$, $H$, the smallest order of their distinguishing  graph is at most $f(R(G),R(H))$?
\medskip

 Finally, we show that two trees of different order are not Ramsey equivalent provided that the Erd\H{o}s-S\'os-Conjecture is true or if one of the trees is balanced.
 We do not know whether there are two Ramsey equivalent non-isomorphic trees on the same number of vertices.
  
%%%%%%%%%%%%%%%%%%%%%%%%%%%%%%%%%%%%%%%
%%%%%%%%    Literature   %%%%%%%%%%%%%
%%%%%%%%%%%%%%%%%%%%%%%%%%%%%%%%%%%%%%%

%  \section{Bibliography}

\newpage

%%%%%%%%%%%%%%%%%%%%%%%%%%%%%%%%%%%%%%%%%%%%%
%%%%%%%%%%%%%%%%%%%%%%%%%%%%%%%%%%%%%%%%%%%%%
%%%%%%%%%%%  APPENDIX   A                                     %%%%%%%%%%%%%%%%
%%%%%%%%%%%%%%%%%%%%%%%%%%%%%%%%%%%%%%%%%%%%%
%%%%%%%%%%%%%%%%%%%%%%%%%%%%%%%%%%%%%%%%%%%%%

\begin{appendix} 

%Adjust counter for lemmas
\edef\currentLemmaCounter{\thelemma}
\setcounter{theorem}{\appendixLemmas}

\section{ Lemmas \ref{lem:minDegG4},  \ref{lem:NoMonoKab},  \ref{lem::minDegG5},  \ref{lem:W4Isolated},  \ref{lem:K23Isolated}.}

\begin{lemma}\label{lem:minDegG4}
  $\cR_{\delta}(Z_4)=1$.
  \end{lemma}
  \begin{proof}
	 Let $\Gamma$ denote the graph obtained from $K_9$ by adding $9$ new vertices and a matching between these and the vertices in $K_9$. 
	 We have $K_9\not\to Z_4$, due to any $4$-factorization, and we shall show $\Gamma\to Z_4$.	
	 Then each minimal Ramsey graph of $Z_4$ contained in $\Gamma$ contains at least one of the vertices of degree $1$ and thus $\cR_{\delta}(Z_4)=1$.
	 
   Consider the copy $K$ of $K_9$ in $\Gamma$ and let $c$ denote a $2$-edge-coloring of $K$ with no monochromatic copy of $Z_4$.
   We will show that there is red and a blue copy of $Z_1$ in $K$ with the same vertex $x$ of degree $4$, see Claim~\ref{G4_Claim2}.
   Then there is a monochromatic copy of $Z_4$ in $\Gamma$ no matter which color is assigned to the edge pendent at $x$.
   Thus $\Gamma\to Z_4$.
   
   \setcounter{claim}{0}
   \begin{claim}\label{G4_Claim1}
    There is no vertex in $K$ with $5$ incident edges of the same color under $c$.
   \end{claim}
   For the sake of contradiction assume $u$ in $K$ has $5$ incident red edges.
   Let $N$ denote the $5$ neighbors of $u$ incident to these edges and $x,y,z$ denote the vertices in $K$ not incident to these edges.
   Then there is at most one red edge between $N$ and each vertex in $\{x,y,z\}$, as otherwise there is a red $Z_4$ in $K$.
   So there are two distinct vertices $v$, $w$ in $N$ such that there are only blue edges between $\{v,w\}$ and $\{x,y,z\}$.
   Since each vertex in $\{x,y,z\}$ is incident to $4$ blue edges to $N$, each of the vertices in $\{x,y,z\}$ is the degree $4$ vertex in a blue copy of $Z_1$ with another vertex from $\{x,y,z\}$ and four vertices from $N$.
   Thus there are only red edges between $u$ and $\{x,y,z\}$ and only red edges within $\{x,y,z\}$, as otherwise there is a blue $Z_4$.
   But then $\{u,x,y,z\}$ forms a red $C_4$ with $3$ red edges pendent at $u$ (those to $N$), a monochromatic $Z_4$, a contradiction.
   This proves Claim~\ref{G4_Claim1}.
   
   \begin{claim}\label{G4_Claim2}
    There is a vertex in $K$ which is the vertex of degree $4$ in a red and a blue copy of $Z_1$ under $c$.
   \end{claim}
   
   By Claim~\ref{G4_Claim1} the red and the blue subgraph of $K$ under $c$ are $4$-regular.
   Consider a vertex $u$ in $K$ and let $N_r$ and $N_b$ denote the sets of neighbors in $K$ adjacent to $u$ via red respectively blue edges.
   If there are vertices $v\in N_b$ and $w\in N_r$ with two red edges between $v$ and $N_r$ and two blue edges between $w$ and $N_b$, then $u$ is the degree $4$ vertex in a red and in a blue copy of $Z_1$ and we are done.
   So we assume that there is at most one blue edge between $N_b$ and each vertex in $N_r$.   
   Since $|N_r|=4$ and the blue subgraph is $4$-regular, each vertex in $N_r$ sends at most $3$ blue edges to the other vertices in $N_r$ and at least one blue edge to $N_b$.
   Hence there is exactly one blue edge between $N_b$ and each vertex in $N_r$ and $N_r$ forms a blue $K_4$.
   If the blue edges between $N_r$ and $N_b$ form a matching, then $N_b$ induces a blue $C_4$ and two independent red edges.
   Then each vertex in $N_b$ is the vertex of degree $4$ in a blue and in a red copy of $Z_1$ and we are done.
   If the blue edges between $N_r$ and $N_b$ do not form a matching then there is vertex $v\in N_b$ with exactly two blue and two red edges between $v$ and $N_r$ (three or four blue edges is not possible).
   Hence $v$ is contained in a blue $C_4$ with three vertices from $N_r$ and there are two blue edges pendent at $v$, one to $u$ and one within $N_b$.
   Moreover $v$ is contained in a red $C_4$ with $u$ and two vertices from $N_r$ and there are two red edges pendent at $v$ within $N_b$.
   This proves Claim~\ref{G4_Claim2}.
%    \medskip
%    
%    As argued above we have $\Gamma\to Z_4$ due to Claim~\ref{G4_Claim2}.
%    Since $K_9\not\to Z_4$ (by any $4$-factorization), we have that each minimal Ramsey graph of $Z_4$ contained in $\Gamma$ contains at least one of the vertices of degree $1$.
%    Hence $\cR_{\delta}(G_4)=1$.
  \end{proof}
  
  Let $Z_5$ denote the graph obtained from $K_{2,3}$ by adding a pendent edge at a vertex of degree $3$, see Figure~\ref{fig:K23Special}.
%   Consider a complete bipartite graph $K$ with partite sets $A$ and $B$ with $|A|=3$ and $|B|=12$.
% Let $S$ denote a set of $2$ vertices in $A$.
% Obtain a graph $G$ from $K$ by adding a set $X$ of three new vertices and the edges $xu$ for each $u\in S$, $x\in X$.
% The following Lemma is based on the proof in \cite{Fox_MinDegreeRamseyMinimal} that $K_{3,13}$ is a minimal Ramsey graph for $K_{2,3}$

\begin{lemma}\label{lem:NoMonoKab}
 In any $2$-edge-coloring of $K_{3,13}-e$ without monochromatic $Z_5$ the vertex of degree $2$ is incident to exactly one red and one blue edge.
\end{lemma}
\begin{proof}
 Let $x$ be the vertex of degree $2$ in $K_{3,13}-e$, let $B$ denote the vertices of degree $3$ and $A=V(K_{3,13}-e)\setminus(B\cup\{x\})$.
 Then $|A|=3$, $|B|=12$ and there is a complete bipartite graph between $A$ and $B$.
 Consider a $2$-edge-coloring of $K_{3,13}-e$ without monochromatic $Z_5$.
 We consider the edges between $A$ and $B$ first.
 Since $|A|=3$, each vertex in $B$ is incident to at least $2$ red or $2$ blue edges by pigeonhole principle.
 Let $B_r\subseteq B$ denote the set of vertices in $B$ incident to at least $2$ red edges and $B_b = B\setminus B_r$ those incident to $2$ blue edges.
 Without loss of generality assume $|B_r|\geq |B_b|$, thus $|B_r|\geq 6$.
 For $v\in B_r$ let $A_v$ denote the vertices in $A$ adjacent to $v$ via red edges.
 Then $|A_v|\geq 2$.
 Assume there is a set $A'\subseteq A$ of size $2$ and distinct vertices $v_1, v_2, v_3\in B_r$ with $A'=A_{v_i}$, $1\leq i\leq 3$.
 Then all edges between $A'$ and $B\setminus \{v_1,v_2,v_3\}$ are blue, as otherwise there is a red $Z_5$.
 But these edges form a blue $K_{2,9}$ which contains $Z_5$, a contradiction.
 Hence for each set $A'\subseteq A$ of size $2$ there are most two vertices $v$ in $B_r$ with $A'=A_{v}$.
 Since $|A|=3$ and $|B_r|\geq 6$, there are exactly $2$ vertices $v\in B_r$ with $A'=A_v$ for each such $A'$ and thus $|B_r| = 6$.
 Hence $|B_b| = |B\setminus B_r| = 6$ and the same arguments applied to $B_b$ and the blue edges show that for each $A'\in \binom{A}{2}$ there are exactly $2$ vertices in $B_b$ adjacent to $A'$ with only blue edges too.
 Now consider the edges incident to $x$.
 If there are $2$ red edges, then together with the neighbors of $x$ and some $3$ vertices from $B_r$ there is a red $Z_5$, a contradiction.
 The same argument holds for the blue edges and hence there is exactly one red and one blue edge incident to $x$.
\end{proof}

\begin{lemma}\label{lem::minDegG5}
 $\cR_{\delta}(Z_5)=1$.
\end{lemma}
\begin{proof}
 Consider the graph $\Gamma$ obtained from a complete bipartite graph on partite sets $X=\{x_1,x_2,x_3\}$ and $W=\{\ell_1,\ell_2,\ell_{3},r_1,r_{2},w\}$ by adding $3$ new vertices and a matching between these and the vertices in $X$.
 See Figure~\ref{fig:K23MinDegree} for an illustration.
 We construct a graph $F$ as follows.
 Let $C$ denote the complete bipartite graph on parts $A$ and $B$ with $|A|=3$ and $|B|=12$.
 For each $i$, $1\leq i \leq 3$, take a copy $C_i$ of $C$ and identify $\{w,\ell_i\}$ with two vertices in the smaller part of $C_i$.
 An illustration is given in Figure~\ref{fig:K23MinDegree}.
 \begin{figure}[htbp]
  \centering
  \includegraphics{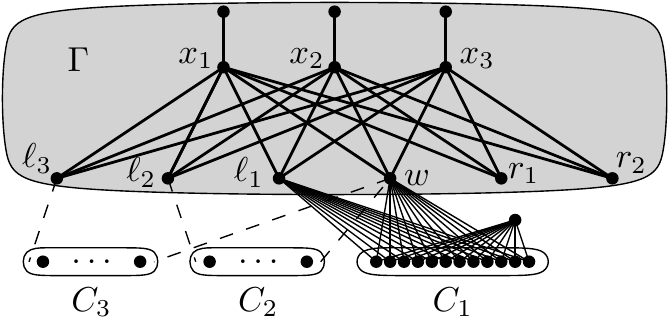}
  \caption{A graph $F$ which is Ramsey for $Z_5$.}
  \label{fig:K23MinDegree}
  \end{figure}
 From now on $\ell_i$ and $w$ refer to the identified vertices.
 We prove $F\to Z_5$ next.
 Consider a $2$-edge-coloring of $F$.
 Then each $x_j\in X$ together with each $C_i$, $1\leq i\leq 3$, induces a copy of $K_{3,13}-e$.
 Hence there is a monochromatic copy of $Z_5$ by Lemma~\ref{lem:NoMonoKab}, if there are $\ell_i$ and $x_j$ such that both edges between $x_j$ and $\{w,\ell_i\}$ are of the same color, $1\leq i,j\leq 3$. 
 So assume that the edges between $x_j$ and $\{w,\ell_i\}$ are of different colors, $1\leq i,j \leq 3$.
 By pigeonhole principle we may assume that there are $2$ vertices in $X$, say $x_1$, $x_2$, such that the edge between $x_j$ and $w$ is red, $j=1,2$.
 Then the edges $x_j\ell_1, x_j\ell_2, x_j\ell_3$ are blue for $j=1,2$.
 Thus there is a blue $K_{2,3}$ between $\{x_1,x_2\}$ and $\{\ell_1, \ell_2, \ell_{3}\}$.
 Hence the edges $x_jr_{1}, x_jr_2$ are red, or otherwise there is a blue $Z_5$.
 Thus there is a red $K_{2,3}$ between $\{x_1,x_2\}$ and $\{w, r_{1}, r_{2}\}$, see Figure~\ref{fig:K23MinDegree_Colored}.
 Altogether there is a monochromatic copy of $Z_5$ no matter which color is assigned to the edge pendent at $x_1$.
 Thus $F\to Z_5$.
  \begin{figure}[htbp]
  \centering
  \includegraphics{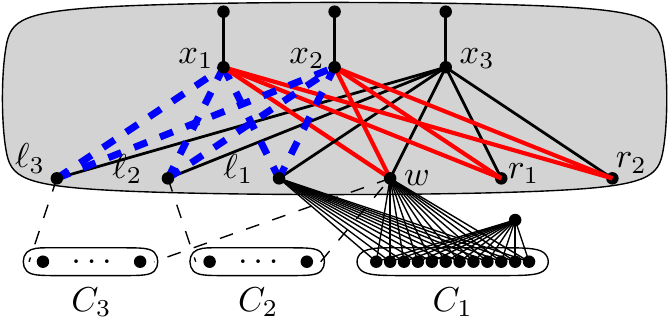}
  \caption{A red and a blue copy of $K_{2,3}$ in a $2$-edge-coloring of $F$ with the same degree $3$ vertices.}
  \label{fig:K23MinDegree_Colored}
 \end{figure}
 
 Let $F'$ denote the graph obtained from $F$ by removing the vertices of degree $1$.
 It remains to show that $F'\not\to Z_5$.
 Then every minimal Ramsey graph of $Z_5$ in $F$ contains at least one of the degree $1$ vertices and hence $\cR_{\delta}(Z_5)=1$.
 Consider the following coloring of $F'$.
 Color all edges between $\{x_1,x_2\}$ and $\{w,r_1,r_2\}$ and all edges between $x_3$ and $\{\ell_1,\ell_2,\ell_3\}$ red.
 Color all other edges between $X$ and $W$ blue.
 Finally color the edges of each $C_i$, $1\leq i\leq 3$, without monochromatic copy of $K_{2,3}$.
 Such a coloring exists since $K_{3,12}\not\to K_{2,3}$~\cite{Fox_MinDegreeRamseyMinimal}.
 Next we show that there is no monochromatic $Z_5$ under this coloring.
 Each copy of $K_{2,3}$ in $F'$ which does not contain any vertex of $X$ is contained in (exactly) one of the $C_i$ and hence is not monochromatic.  
 Moreover each copy of $K_{2,3}$ in $F'$ which contains a vertex from $X$ and a vertex from $V(C_i)\setminus\{w,\ell_i\}$ for some $i$, contains $w$ and $\ell_i$ and hence a red and a blue edge.
 Thus the only monochromatic copies of $K_{2,3}$ in $F'$ are contained in $\Gamma$.
 The part on $2$ vertices of all monochromatic copies of $K_{2,3}$ in $\Gamma$ is contained in $X$.
 Hence there is no monochromatic copy of $Z_5$ because each $x \in X$ is incident to exactly $3$ red and $3$ blue edges.
\end{proof}

The wheel $W_4$ is the graph on $5$ vertices obtained from a cycle $C_4$ of length $4$ by adding a new vertex adjacent to all vertices of the cycle.

\begin{lemma}\label{lem:W4Isolated}
 If $H$ is connected and has at least $6$ vertices, then $H\not\req W_4$.
\end{lemma}
\begin{proof}
 We assume $\omega(H)=3=\omega(W_4)$ due to Lemma~\ref{lem::RamseyGraphSameCliqueNo} and $\chi(H)\geq \chi(W_4)=3$ due to Observation~\ref{obs::bipartite}.
 Let $\eps=2^{-5}$ and let $F$ be a graph with $F\epsto C_4$ and $\omega(F) = 2$, which exists by Lemma~\ref{lem::epsRamsey}.
 We construct a graph $\Gamma$ by taking the vertex disjoint union of $F$ and a copy $K$ of $K_5$ and placing a complete bipartite graph between $F$ and $K$.
 We shall show that $\Gamma \to W_4$, but $\Gamma \not\to H$. 
 
 Color all edges within $F$ and within $K$ red and all other edges blue.
 Since $\omega(F) = 2 < \omega(H)$, $H\not\subseteq F$.
 Since $|V(H)| \geq 6$, $H\not\subseteq K$.
 Since $H$ is connected there is no red copy of $H$.
 The blue subgraph is a complete bipartite graph and $\chi(H) \geq 3$.
 Thus there is no blue copy of $H$.

 It remains to show that $\Gamma\to W_4$.
 Consider a $2$-edge-coloring of $\Gamma$.
 By the Focusing Lemma (Lemma~\ref{lem::focusing}) there is a set $V$ of $2^{-5}|V(F)| = \epsilon |V(F)|$ vertices in $F$ such that between $V$ and each vertex in $K$ all edges are of the same color.
 Since $F\epsto C_4$ there is a monochromatic  copy $C$ of $C_4$ in $F[V]$.
 Assume without loss of generality that $C$ is blue. 
 If there is a vertex in $K$ which sends a blue star to $C$ then there is a blue $W_4$ and we are done.
 So assume all vertices in $K$ send red stars to $C$.
 If there is no blue copy of $W_4$ in $K$, then there are two adjacent red edges $e$ and $f$ in $K$ (since the complement of $W_4$ in $K_5$ is a maximum matching).
 Then $e$, $f$ and any two vertices from $C$ form a red copy of $W_4$, with the vertex of degree $4$ being the common vertex of $e$ and $f$, see Figure~\ref{fig::W4Gamma}.
 Hence $\Gamma\to W_4$.
 \begin{figure}[htbp]
 \centering
 \includegraphics{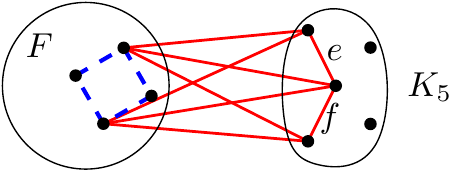}
 \caption{A red (solid lines) copy of $W_4$ in some $2$-edge-coloring of $\Gamma$.}
 \label{fig::W4Gamma}
\end{figure}
\end{proof}

\begin{lemma}\label{lem:K23Isolated}
 If $H$ is connected and has at least $6$ vertices, then $H\not\req K_{2,3}$.
\end{lemma}
\begin{proof}
 We assume that $H$ is bipartite and contains a cycle by Observation~\ref{obs::bipartite} and Lemma~\ref{lem::extremalNo}, respectively.
 Since $K_{3,13}\to K_{2,3}$~\cite{Fox_MinDegreeRamseyMinimal} we assume $K_{3,13}\to H$, since otherwise $H\not\req K_{2,3}$.
 Hence one of the partite sets of $H$ contains at most $2$ vertices, since otherwise coloring the edges of $K_{3,13}$ with a red copy of $K_{2,13}$ and an edge disjoint blue copy of $K_{1,13}$ does not yield a monochromatic copy of $H$.   
 So $H$ is $K_{2,b}$ with some edges pendent at the part of size $2$ for some $b\in\NN$ with $b\geq 2$.
 
 Suppose $H$ has at least $8$ vertices.
 We claim $R(H)>10$.
 Indeed, color the edges of a $K_{5,5}$ in $K_{10}$ red and all other edges blue.
 Then the blue edges form vertex disjoint copies of $K_5$ and do not contain $H$ since $H$ is connected.
 The red subgraph does not contain $H$ because one of the bipartition classes of $H$ has at least $6$ vertices (as the other has only $2$).
 Hence $H\not\req K_{2,3}$ because $R(K_{2,3})=10$~\cite{hendry_Ramsey5Vertices}.
 
 Suppose $|V(H)|\in\{6,7\}$.
 Then $H$ is isomorphic to one of the graphs $Z_i$ given in Figure~\ref{fig:K23Special} or a supergraph of $Z_5$.
 Note that $Z_5$ contains $K_{2,3}$.
 We have $R(Z_1)=7$, $R(Z_2)=8$ by~\cite{Ramsey6Edges} and $R(Z_3)= 9$, $R(Z_4)=10$ by~\cite{Ramsey7Edges}.
 Thus $H\not\req K_{2,3}$ if $H$ is isomorphic to one of the graphs $Z_i$, $1\leq i\leq 3$, because $R(K_{2,3})=10$.
 Moreover $H\not\req K_{2,3}$ if $H$ is isomorphic to $Z_4$ or $Z_5$ because $\cR_{\delta}(Z_4)=\cR_{\delta}(Z_5)=1$ by Lemma~\ref{lem:minDegG4} respectively Lemma~\ref{lem::minDegG5}, but $\cR_{\delta}(K_{2,3})\geq \delta(K_{2,3})=2$.
 So assume $H$ is a supergraph of $Z_5$.
 Let $F$ denote a minimal Ramsey graph of $Z_5$ with $\delta(F)=1$ and obtain a graph $F'$ by removing a vertex of degree $1$ from $F$.
 Then $F\to K_{2,3}$ because $K_{2,3}\subseteq Z_5$ and $F'\to K_{2,3}$ (since $F$ is not minimal Ramsey for $K_{2,3}$). But $F'\not\to Z_5$, thus $F'\not\to H$, and hence $H\not\req K_{2,3}$.
\end{proof}

% reset the counter for new lemmas
\setcounter{theorem}{\currentLemmaCounter}

%%%%%%%%%%%%%%%%%%%%%%%%%%%%%%%%%%%%%%%%%%%%%
%%%%%%%%%%%%%%%%%%%%%%%%%%%%%%%%%%%%%%%%%%%%%
%%%%%%%%%%%  APPENDIX   B                                       %%%%%%%%%%%%%%%%
%%%%%%%%%%%%%%%%%%%%%%%%%%%%%%%%%%%%%%%%%%%%%
%%%%%%%%%%%%%%%%%%%%%%%%%%%%%%%%%%%%%%%%%%%%%

\section{Small Distinguishing Graphs}\label{sec::ProofsSmallLemmas}

\begin{lemma}\label{lem::H52P4e}
$H_{5,2}\to P_4+e$.
\end{lemma}

\begin{proof}
Let $u$ denote the vertex of degree $2$ in $H_{5,2}$ and let let $e=vw$ denote the edge incident to both neighbors of $u$. 
Let $x$, $y$, $z$ denote the other vertices.
% See Figure~\ref{fig::H5x} for an illustration.
Assume there is a $2$-edge-coloring of $H_{5,2}$ without monochromatic $P_4+e$.
Without loss of generality assume the edge $uv$ is blue.

If (\textbf{Case 1}) $e$ and two more edges $vx$, $vy$ incident to $v$ are red, then either there is a red $P_4+e$ containing these edges or $zx$, $zy$, $zw$ and $uw$ are blue. But then these form a blue copy of $P_4+e$.

If (\textbf{Case 2}) $e$ and at most one other edge incident to $v$ is red, then assume $vx$, $vy$ are blue.
Then the edges $wx$, $wy$, $zx$, $zy$ and $uw$ must be red, but these contain a copy of $P_4+e$.

If (\textbf{Case 3}) $e$ and $vx$ are blue, then $wy$, $wz$, $xy$ and $xz$ must be red.
But then the edges $vy$, $vz$ and $uw$ must be blue, but contain a copy of $P_4+e$.

If (\textbf{Case 4}) $e$ is blue but all the edges $vx$, $vy$, $vz$ are red, then $wx$, $wy$, $wz$ are blue.
But together with $uv$ and $vw$ this is a blue $P_4+e$.
\end{proof}

\begin{lemma}\label{lem::H53P5}
$H_{5,3}\to P_5$.
\end{lemma}

\begin{proof}
 Assume there is $2$-edge-coloring of $H_{5,3}$ without monochromatic $P_5$.
 Let $u$ denote the vertex of degree $3$ in $H_{5,3}$, let $x$, $y$, $z$ denote its neighbors and $v$, $w$ the remaining vertices.
 There are two edges of the same color incident to $u$, assume $ux$, $uy$ are red.
 Since there is no monochromatic $K_{2,3}$ (it contains $P_5$) there is at least one edge from $\{x,y\}$ to $\{v,w,z\}$ in red.
  
 If (\textbf{Case 1}) there are $r,r'\in\{v,w,z\}$ such that the edges $xr$, $yr'$ are red, then $r=r'$.
 Then all edges from $\{x,y\}$ to $\{v,w,z\}\setminus\{r\}$ are blue.
 But then the edge from $r$ to $\{v,w,z\}\setminus\{r\}$ can be neither red nor blue.
 
 If (\textbf{Case 2}), without loss of generality, $x$ has only blue edges to $\{v,w,z\}$, then $yp$ is red for some $p\in\{v,w,z\}$ and all edges from $p$ to $\{v,w,z\}\setminus\{p\}$ are blue.
 This yields a blue $C_4$ on $\{x,v,w,z\}$.
 Since all edges which are incident to this $C_4$ (but not contained) are red we can find a red $P_5$.
\end{proof}

\begin{lemma}\label{lem::K55C4}
$K_{5,5}\to C_4$.
\end{lemma}
\begin{proof}
Consider a $2$-edge-coloring of the edges of $K_{5,5}$ and vertex $v$.
Let $V$ denote the partite set of $K_{5,5}$ containing $v$.
Then $v$ is incident to three edges $vx$, $vy$, $vz$ of the same color, say red.
From each of the four vertices in $V\setminus\{v\}$ at most one edge to $\{x,y,z\}$ is red, otherwise there is a red $K_{2,2}$.
But then there are two vertices in $\{x,y,z\}$ and two vertices in $V\setminus\{v\}$ forming a blue $K_{2,2}$ by pigeonhole principle.
\end{proof}

\begin{figure}[htbp]
 \centering
 \includegraphics{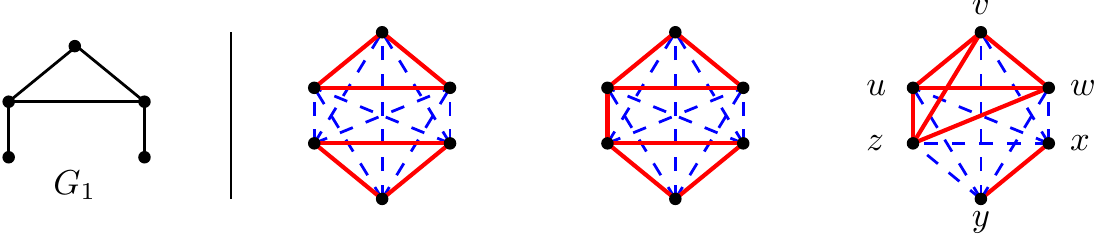}
 \caption{All possible $2$-colorings of $K_6$ without monochromatic $G_1$.}
 \label{fig::K6G1}
\end{figure}

\begin{lemma}\label{lem::K6G1}
 Each $2$-edge-coloring of $K_6$ without monochromatic $G_1$ equals to one of the colorings given in Figure~\ref{fig::K6G1}, up to isomorphism and renaming colors.
\end{lemma}
\begin{proof}
 Consider a $2$-edge-coloring of $K_6$ on vertices $u,v,w,x,y,z$ without monochromatic $G_1$.
 We may assume that $K=\{u,v,w\}$ forms a red $K_3$ since $R(K_3)=6$.
 Clearly there are no two independent red edges between $K$ and $K^c=\{x,y,z\}$.
 If there is a vertex in $K$ incident to at least two red edges to $K^c$, then all edges between the two other vertices in $K$ and $K^c$ are blue.
 Thus these blue edges form a blue $K_{2,3}$, and hence no edge in $K^c$ is blue.
 But if all edges in $K^c$ are red, then there is a red $K_3$ with two vertices in $K^c$ and one vertex in $K$ and a pending red edge in $K$ and a pending in $K^c$.
 So we may assume that at most one vertex in $K^c$ is adjacent to $K$ in red, say $z$.
 Then $\{x,y\}$ and $K$ form a blue $K_{2,3}$.
 This shows that $xy$ is red. 
 We consider the cases how many edges between $x,y$ and $z$ are red.
 
 If (\textbf{Case 1}) $xy$ is the only red edge, then $\{x,y\}$ and $\{u,v,w,z\}$ induce a blue $K_{2,4}$ and any additional blue edge within this $K_{2,4}$ yields a blue $G_1$.
 Hence the red edges form a $K_4$ plus disjoint $K_2$, which corresponds to the rightmost coloring of Figure~\ref{fig::K6G1}.
 
 If (\textbf{Case 2}) there are at least two red edges, then $z$ is not part of any red $K_3$ on $\{u,v,w,z\}$, since this $K_3$ would have a red pending edge in $K$ and another one in $K^c$.
 Hence there is at most one red edge from $z$ to $K$.
 Thus there is a blue copy of $K_{3,3}-e$ between $K$ and $K^c$.
 This shows that no edge in $K^c$ is blue and the coloring corresponds to the left or the middle coloring in Figure~\ref{fig::K6G1}.
\end{proof}

\begin{lemma}\label{lem::gammaG1}
 $\Gamma\to G_1$.
\end{lemma}
\begin{proof}
 Assume $c$ is a $2$-coloring of the edges of $\Gamma$, labeled like in Figure~\ref{fig::graphGamma}, without monochromatic $G_1$.
 Let $K$ denote the copy of $K_6$ in $\Gamma$.
 Due to Lemma~\ref{lem::K6G1}, $c$ restricted to $K$ is isomorphic to one of three colorings of $K_6$ given in Figure~\ref{fig::K6G1}.
 We will distinguish cases based on the coloring of $K$ under $c$.
 
 (\textbf{Case 1}:) The red subgraph of $K$ under $c$ consists of two disjoint $K_3$ and the blue edges in $K$ form a copy of $K_{3,3}$.
 If one of these blue edges is contained in a blue $K_3$ with a vertex from $\{x,y,z\}$, then there is blue $G_1$.
 Due to the construction of $\Gamma$ we can find three vertex disjoint copies of $K_3$ each with exactly one vertex from each of the red $K_3$ in $K$ and exactly one vertex from $\{x,y,z\}$.
 Since there is a red edge from $K$ to $\{x,y,z\}$ in each of these, one of the red $K_3$ in $K$ has two independent pending red edges.
 This gives a red $G_1$, a contradiction.
 
 (\textbf{Case 2}:) The red subgraph of $K$ under $c$ consists of two disjoint $K_3$ connected by a single edge $e$.
 Then all edges from $K$ to $\{x,y,z\}$ are blue if not adjacent to $e$.
 Then there are two vertices in $K$, not incident to $e$, each having two blue edges to $\{x,y,z\}$ but only one common neighbor in $\{x,y,z\}$.
 Since they are connected by an blue edge in $K$, this gives a blue $G_1$, a contradiction.
 
 (\textbf{Case 3}:) The red subgraph of $K_6$ consists of $K_4$ and a single edge $e$.
 Then all edges from this $K_4$ to $\{x,y,z\}$ are blue.
 The blue edges in $K$ form $K_{2,4}$.
 Again, if one of these blue edges in $K$ forms a blue triangle with a vertex from $\{x,y,z\}$, then there is a blue $G_1$.
 Thus all edges from $e$ to $\{x,y,z\}$ are red.
 If $e\not\in\{v_1v_2,u_1u_2,w_1w_2\}$, then $e$ together with $\{x,y,z\}$ forms a red copy of $G_1$.
 So assume $e=w_1w_2$.
 If the edge $xy$ is blue, then $\{x,y,u_1,u_2,v_1\}$ gives a blue $G_1$.
 If it is red, then $\{x,y,z,w_1,w_2\}$ gives a red $G_1$, a contradiction.
 
 Altogether we proved that there is no $2$-edge-coloring of $\Gamma$ without monochromatic $G_1$.
\end{proof}

\begin{figure}[b]
 \centering
 \includegraphics{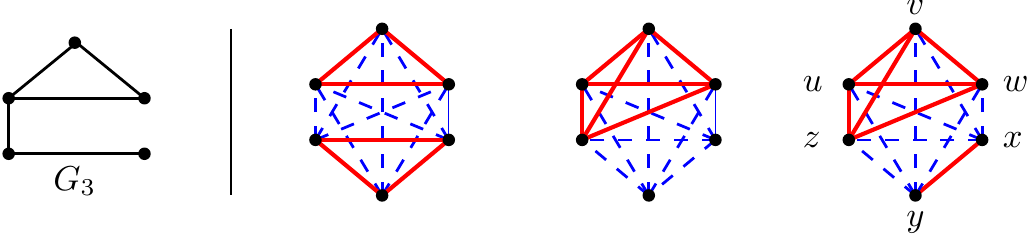}
 \caption{All possible $2$-colorings of $K_6$ without monochromatic $G_3$.}
 \label{fig::K6G3}
\end{figure}

\begin{lemma}\label{lem::K6G3}
 Each $2$-edge-coloring of $K_6$ without monochromatic $G_3$ equals to one of the colorings given in Figure~\ref{fig::K6G3}, up to isomorphism and renaming colors.
\end{lemma}
\begin{proof}
 Consider a $2$-edge-coloring of $K_6$ on vertices $u,v,w,x,y,z$ without monochromatic $G_3$.
 We may assume that $K=\{u,v,w\}$ forms a red $K_3$.
 
 If all edges from $K$ to $K^c$ are blue, then no edge among $K^c$ is blue.
 Thus the coloring corresponds to the left one in Figure~\ref{fig::K6G3}.
 So assume the edge $uz$ is red.
 If $\{u,v,w,z\}$ forms a red $K_4$, then all edges from this $K_4$ to $\{x,y\}$ are blue.
 No matter which color is assigned to $xy$, the coloring has no monochromatic $G_3$ and corresponds to the middle or right color in Figure~\ref{fig::K6G3}.
 
 So assume further that $\{u,v,w,z\}$ is not a red $K_4$ (but $uz$ is still red), without loss of generality $wz$ is blue.
 Since $uz$ is red, $xz$ and $yz$ are blue.
 
 If (\textbf{Case 1}) $wx$ is blue, then $\{w,x,z\}$ is a blue $K_3$ with pending blue edge $yz$.
 Thus $uy$, $vy$ are red.
 But then $wy$ needs to be blue (otherwise $\{v,w,y\}$ is red $K_3$ with pending red path $vuz$) and there is a blue $K_4$.
 Thus the coloring corresponds to the middle or right coloring of Figure~\ref{fig::K6G3} with switched colors, as argued above.
 
 If (\textbf{Case 2}) $wx$ is red, then $xy$ is blue and $vx$, $vz$ are blue.
 Then $vy$ needs to be red, since $\{v,x,y\}$ is a blue $K_3$ with pending blue path $xzw$ otherwise.
 Then $wy$ is blue, since otherwise $\{v,w,y\}$ is a red $K_3$ with pending red path $vuz$ otherwise.
 But now $\{w,y,z\}$ is a blue $K_3$ with pending blue path $zxv$, a contradiction.
\end{proof}

\begin{lemma}\label{lem::gammaG3}
 $\Gamma\to G_3$.
\end{lemma}
\begin{proof}
 Assume $c$ is a $2$-coloring of the edges of $\Gamma$, labeled like in Figure~\ref{fig::graphGamma}, without monochromatic $G_3$.
 Let $K$ denote the copy of $K_6$ in $\Gamma$.
 Due to Lemma~\ref{lem::K6G3}, $c$ restricted to $K$ is isomorphic to one of three colorings of $K_6$ given in Figure~\ref{fig::K6G3}.
 We will distinguish cases based on the coloring of $K$ under $c$.
 
 (\textbf{Case 1}:) The red subgraph of $K_6$ consists of two disjoint $K_3$.
 Then the blue edges in $K$ form a copy of $K_{3,3}$.
 If one of these blue edges is contained in a blue $K_3$ with a vertex from $\{x,y,z\}$, then there is blue $G_3$.
 On the other hand no vertex in $K^c=\{x,y,z\}$ sends a red edge to each of the red $K_3$s in $K$.
 Since there are $4$ edges from each vertex in $\{x,y,z\}$ to $K$, each is incident to at least one red edge and one blue edge.
 If one of the edges induced by $\{x,y,z\}$ is red, then there is a red $G_3$ with a red $K_3$ from $K$ and an edge between them.
 So $\{x,y,z\}$ induces a blue $K_3$ which forms a blue $G_3$ with an edge to $K$ and another contained in $K$.
 
 (\textbf{Case 2}:) The red subgraph of $K_6$ consists of a red $K_4$ only.
 Then no edge incident to this $K_4$ is red.
 Let $a$, $b$ denote the vertices in $K$ not contained in the red $K_4$.
 Then every blue edge between $\{x,y,z\}$ and the red $K_4$ is part of a blue $G_3$ together with $a$, $b$, and another vertex in $K$.
 
 (\textbf{Case 3}:) The red subgraph of $K_6$ consists of a red $K_4$ and a disjoint red edge $e$.
 Again all edges incident to the red $K_4$ are blue and no blue edge in $K$ is contained in a blue $K_3$ with a vertex from $K^c$.
 Thus all edges from $e$ to $K^c$ are red.
 Assume $e\not\in\{v_1v_2,u_1u_2,w_1w_2\}$, say it is $v_2w_2$.
 If $xy$ is blue then $\{x,y,u_1,v_1,z\}$ forms a blue $G_3$.
 If $xy$ is red then $\{x,y,v_2,w_2,z\}$ gives a red $G_3$.
 So assume $e=w_1w_2$.
 If the edge $xy$ is blue, then $\{x,y,u_1,v_1,z\}$ forms a blue $G_3$.
 If it is red, then $\{x,y,z,w_1,w_2\}$ gives a red $G_3$.
\end{proof}

\begin{lemma}\label{lem::K8G123}
 For each $i\in\{1,2,3\}$ a $2$-coloring of $K_8$ does not have a monochromatic $G_i$, if and only if one of the color classes induces two vertex disjoint copies of $K_4$.
\end{lemma}
\begin{proof}
 First of all note that an $2$-edge-coloring of $K_8$ with one color class inducing two vertex disjoint $K_4$'s does not contain a monochromatic copy of $G_i$ for each $i\in\{1,2,3\}$.
 \medskip
 
 On the other hand, consider an arbitrary $2$-edge-coloring of $K_8$ without a monochromatic copy of $G_i$ for a fixed $i\in\{1,2,3\}$.
 There is a monochromatic copy $K$ of $H_{3,1}$, say in red (i.e. a red copy of $K_3$ with a pending edge), since $R(H_{3,1}) = 7$.
 \medskip
 
 Suppose there is no monochromatic $G_1$.
 Then none of the two vertices of degree $2$ in $K$ is incident to another red edge in $K_8$.
 Thus there is a blue copy of $K_{2,4}$.
 Then the part with four vertices in this $K_{2,4}$ contains no further blue edge and induces a red $K_4$.
 But then no edge incident to this $K_4$ is red, and there is a blue copy of $K_{4,4}$.
 Since no other edge might be blue then, there are two disjoint red $K_4$.
 \medskip
 
 Suppose there is no monochromatic $G_2$.
 The vertex of degree $3$ in $K$ has no other incident red edge.
 So it is the center of a blue $K_{1,4}$.
 The degree $1$ vertices in this copy of $K_{1,4}$ do not induce a blue edge, so they induce a red $K_4$.
 But then no edge incident to this $K_4$ is red and there is a blue $K_{4,4}$ between $K$ and the other vertices.
 As argued above the red edges form two disjoint $K_4$s and the blue edges form $K_{4,4}$.
 \medskip
 
 Suppose there is no monochromatic $G_3$.
 Let $K^c$ denote the set of vertices not in $K$.
 Then any edge connecting the vertex $v$ of degree $1$ in $K$ to a vertex in $K^c$ is blue.
 Assume there is a red edge from $K$ to a vertex $u\in K^c$.
 Then any edge connecting $u$ to a vertex in $K^c\setminus\{u\}$ is blue.
 Then each edge $e$ within $K^c\setminus\{u\}$ or from $K^c\setminus\{u\}$ to $K\setminus\{v\}$ is red, since otherwise there is a blue $G_3$ spanned by $u$, $v$ and $e$.
 But then there is red $G_3$, a contradiction.
 
 So all edges between $K$ and $K^c$ are blue.
 Then there is no other blue edge and the red edges form two disjoint copies of $K_4$.
\end{proof}

\begin{lemma}\label{lem::H85G123}
 $H_{8,5}\to G_i$ for each $i\in\{1,2,3\}$.
\end{lemma}
\begin{proof}
Assume there is a $2$-edge-coloring of $H_{8,5}$ without monochromatic $G_i$ for some $i\in\{1,2,3\}$.
We may assume that within the copy of $K_8$ the red edges form two disjoint $K_4$ with a blue $K_{4,4}$ in-between by Lemma~\ref{lem::K8G123}.
Let $v$ denote the vertex of degree $5$.
Then $v$ has only blue incident edges since every neighbor of $v$ is part of a red $K_4$.
But $v$ has a neighbor in each of the red $K_4$'s.
Thus $v$ together with these two vertices forms a blue $K_3$ which is contained in a blue copy of $G_i$ for all $i\in\{1,2,3\}$, a contradiction.
\end{proof}

\begin{lemma}\label{lem::K7G12}
 A $2$-edge-coloring of $K_7$ does not have monochromatic $G_1$, if and only if one of the color classes induces vertex disjoint copies of $K_3$ and $K_4$.
\end{lemma}
\begin{proof}
 The proof is very similar to the proof of Lemma~\ref{lem::K8G123}.
\end{proof}

\begin{lemma}\label{lem::GammaPrimeG1}
 $\Gamma'\to G_1$.
\end{lemma}
\begin{proof}
 Assume there is a $2$-edge-coloring of $\Gamma'$ without monochromatic $G_1$.
 Due to Lemma~\ref{lem::K7G12} we may assume that the copy of $K_7$ in $\Gamma'$ is colored such that the blue edges induce a copy of $K_{3,4}$ and red consists of two disjoint copies of $K_4$ and $K_3$.
 Let $K$ denote the red $K_3$ and $u$, $v$ the two vertices of degree $5$ in $\Gamma'$.
 Then each edge from $\{u,v\}$ to the red $K_4$ is blue and there are at least two such edges incident to each of $u$, $v$.
 Thus each edge from $u$ or $v$ to $K$ is red, since there is a blue $G_1$ otherwise.
 Due to construction of $\Gamma'$ there are two independent edges from $K$ to $\{u,v\}$ and thus a red $G_1$, a contradiction.
\end{proof}

\begin{conjecture}\label{lem::K8G4}
 A $2$-edge-coloring of $K_8$ does not have monochromatic $G_4$, if and only if one of the color classes induces two disjoint copies of $K_4$ with at most one edge of same color in-between (i.e. the other color spans $K_{4,4}$ or $K_{4,4}-e$).
\end{conjecture}
% \begin{conjecture}
%  Clearly both colorings described in the lemma do not have monochromatic $G_4$.
%  Consider an arbitrary $2$-coloring of $K_8$ without monochromatic $G_4$.
%  
%  First of all assume there is a red copy $K$ of $K_4$.
%  Let $K^c$ denote the vertices not in $K$.
%  Then no vertex in $K^c$ is incident to two red edges to $K$ as they form a red $G_4$ with the vertices in $K$.
%  A case distinction shows that no edge in $K^c$ is blue, since there is a blue $G_4$ otherwise.
%  Hence $K^c$ is a red $K_4$ and no vertex in $K$ is incident to two red edge to $K^c$.
%  Thus there is at most one red edge between $K$ and $K^c$.
%  \medskip
%  
%  It remains to show that there is a monochromatic $K_4$.
%  Assume not.
%  By Lemma~\ref{lem::K8G123} there is a monochromatic copy of $G_3$, say in red.
%  Let $K$ denote the set of vertices of degree at least $2$ in this $G_4$.
%  By assumption at least one of the remaining edges in $K$ is blue.
%  
%  \textbf{Big case analysis following, ToDo}
% \end{proof}

\begin{lemma}\label{lem::H86G4}
 If Conjecture~\ref{lem::K8G4} holds, then $H_{8,6}\to G_4$.
\end{lemma}
\begin{proof}
 Assume there is a $2$-edge-coloring of $H_{8,6}$ without monochromatic $G_4$.
 We assume that the coloring of $K_8$ in $H_{8,6}$ has two disjoint red $K_4$ connected by at most one red edge according to Conjecture~\ref{lem::K8G4}.
 Let $v$ denote the vertex of degree $6$.
 It is incident to at most one red edge to to each of the red $K_4$.
 Thus there is a blue $K_3$ with $v$ and one vertex from each red $K_4$.
 But this forms a blue $G_4$ together with some of the other blue edges, a contradiction.
\end{proof}

\begin{lemma}\label{lem::K1212}
 $K_{12,12}\to K_{2,3}$.
\end{lemma}
\begin{proof}
 Consider a $2$-coloring of the edges of $K_{12,12}$ and vertex $u$.
Let $V$ denote the partite set of $K_{12,12}$ containing $u$.
Then $u$ is incident to five edges $uv$, $uw$, $ux$, $uy$, $uz$ of the same color, say red.
From each of the $11$ vertices in $V\setminus\{u\}$ at most two of the edges to $\{v,w,x,y,z\}$ are red, otherwise there is a red $K_{2,3}$.
This means that there are three blue edges between each of the vertices in $V\setminus\{u\}$ and $\{v,w,x,y,z\}$.
There are $10$ sets of size~$3$ in $\{v,w,x,y,z\}$ and $11$ vertices in $V\setminus\{v\}$.
Hence there are two vertices in $V\setminus\{v\}$ and three vertices in $\{v,w,x,y,z\}$ forming a blue $K_{2,3}$ by pigeonhole principle.\end{proof}

% \begin{figure}[htbp]
%  \begin{minipage}{0.48\textwidth}
%  \centering
%  \includegraphics[height=3.5cm]{H96H3}
%  \caption{A coloring of $H_{9,6}$ without monochromatic $H_3$ (all edges not drawn in the grid are blue).}
%  \label{fig::H96H3}
%  \end{minipage}
%  \hfill
%  \begin{minipage}{0.48\textwidth}
%  \centering
%  \includegraphics[height=3.5cm]{H96H4}
%  \caption{A coloring of $H_{9,6}$ without monochromatic $H_4$ (all edges not drawn in the grid are blue).}
%  \label{fig::H96H4}
%  \end{minipage}
% \end{figure}

\begin{lemma}\label{lem::K9H32}
 A $2$-edge-coloring of $K_9$ does not have a monochromatic copy of $H_{3,2}=K_4-e$ if and only if each color class is isomorphic to the Cartesian product $K_3\times K_3$.
\end{lemma}
\begin{proof}
 First of all observe that $K_3\times K_3$ does not contain a copy of $H_{3,2}$ since every edge is contained in exactly one copy of $K_3$.
 Moreover the complement of $K_3\times K_3$ (as a subgraph of $K_9$) is isomorphic to $K_3\times K_3$.
 Hence the edges of $K_9$ can be $2$-colored without monochromatic $H_{3,2}$ using two edge disjoint copies of $K_3\times K_3$.

 On the other hand consider a $2$-edge-coloring $c$ of $K_9$ without monochromatic $H_{3,2}=K_4-e$.
 We will assign labels $v_{i,j}$, $1\leq i,j\leq 3$, to the vertices of $K_9$ such that this labeling corresponds to an arrangement of the vertices in a $3\times 3$ grid where the red subgraph spans all rows and columns and all other edges are blue.
%  It is not hard to see that the blue edges form a copy of $K_3\times K_3$ then as well.
 
 There is a monochromatic copy of $G_5$ under $c$, say in red, since $R(G_5)=9$, see Figure~\ref{fig::5Vert}.
 Let $K=\{v_{1,1}$, $v_{1,2}$, $v_{1,3}$, $v_{2,1}$, $v_{3,1}\}$ denote the vertices of this $G_5$ such that $v_{1,1}$ is the vertex of degree $4$ and the edges of this red $G_5$ span the first row and first column in the grid, see Figure~\ref{fig::K9H32}.
 Observe that no edge spanned by $K$ is red except for the edges in the red $G_5$.
 Indeed if another edge is red, then there is a red $H_{3,2}$ in $K$.
 Let $K^c$ denote the vertices not in $K$.
 We will use the following claim.

 \begin{figure}[tb]
  \begin{minipage}{0.48\textwidth}
  \centering
   \includegraphics{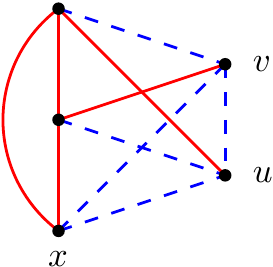}
  \caption{The unique $2$-coloring of $K_5$ 	(up to isomorphism) without monochromatic $H_{3,2}$ provided there is a red $K_3$ (left) and a disjoint blue edge (right).}
  \label{fig::MonoK3Edge}
  \end{minipage}
  \hfill
  \begin{minipage}{0.48\textwidth}
  \centering
   \includegraphics{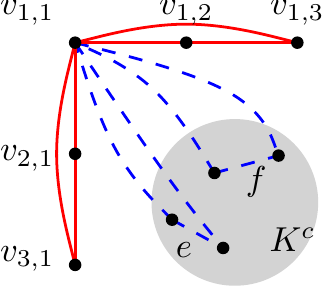}
  \caption{The partial labeling of vertices of $K_9$ under a $2$-coloring without monochromatic $H_{3,2}$ in the proof of Lemma~\ref{lem::K9H32}, with solid red and dashed blue edges.}
  \label{fig::K9H32}
  \end{minipage}
 \end{figure}
 
 \setcounter{claim}{0}
 \begin{claim}\label{claim-1}
  If $C$ is a red copy of $K_3$ and $uv$ is a vertex disjoint blue edge, then there is a vertex $x$ in $C$ such that $xu$ and $xv$ are blue, and there are two independent red and two independent blue edges between $C-x$ and $uv$.
 See Figure~\ref{fig::MonoK3Edge} for an illustration.
 \end{claim}
 
 Indeed, there is at most one red edge between each vertex in $\{u,v\}$ and $C$ and for at most one vertex in $C$ both edges to $uv$ are blue.
 Hence for exactly one vertex in $C$ both edges to $uv$ are blue and there are exactly two further independent blue edges between $C$ and $uv$.
 This proves Claim~\ref{claim-1}.
 
 By assumption the four vertices in $K^c$ do not induce a monochromatic $H_{3,2}$ and hence there are at least two blue edges $e$, $f$.
 Let $C_1$, $C_2$ denote the red copies of $K_3$ in $K$.
 We will apply Claim~\ref{claim-1} to each of the pairs $\{e,C_1\}$, $\{e,C_2\}$, $\{f,C_1\}$, $\{f,C_2\}$.
 There is a vertex $x_i$ in $C_i$, $i=1,2$, such that both edges between $x_i$ and a blue edge in $K^c$ are blue by Claim~\ref{claim-1}.
 Then $x_1=x_2=v_{1,1}$, since otherwise there is blue copy of $H_{3,2}$.
 Hence the blue edges in $K^c$ are independent, since two adjacent blue edges together with $v_{1,1}$ form a blue $H_{3,2}$.
 Thus $e$ and $f$ are the only blue edges in $K^c$.
 See Figure~\ref{fig::K9H32} for the partial labeling.
 Furthermore there are two independent red edges and two independent blue edges from each of the edges $e$ and $f$ to each $C_i-v_{1,1}$, $i=1,2$, by Claim~\ref{claim-1}.
 It remains to find labels for the vertices in $e$ and $f$.
 
 \begin{claim}\label{claim-2}
  For any two vertices $u\in \{v_{1,2},v_{1,3}\}$, $v\in \{v_{2,1},v_{3,1}\}$ there is exactly one vertex $w$ in $K^c$ such that $uw$ and $vw$ are red.
 \end{claim}

 Indeed, assume there are two such vertices $w, w'$ in $K^c$ for some pair $u$,$v$.
 Then the edge $ww'$ is red by Claim~\ref{claim-1} and there is a red $H_{3,2}$.
 Thus there is at most one such vertex.
 Assume there is no such vertex in $K^c$ for some pair.
 Then there is a red $H_{3,2}$, since there are two independent red edges between each of $e$ and $f$ and each $C_i$, $i=1,2$, by Claim~\ref{claim-1}, a contradiction.
 This proves Claim~\ref{claim-2}.
 
 Let $v_{2,2}$ denote the vertex which is adjacent to $v_{1,2}$ and $v_{2,1}$ in red which exists by Claim~\ref{claim-2}.
 Without loss of generality assume $v_{2,2}$ is incident to $e$.
 Let $v_{3,3}$ denote the other vertex incident to $e$.
 Due to Claim~\ref{claim-1} applied to $e$ and $C_1$ and $C_2$, the edges $v_{3,3}v_{1,3}$ and $v_{3,3}v_{3,1}$ are red and the edges $v_{2,2}v_{1,3}$, $v_{2,2}v_{3,1}$, $v_{3,3}v_{1,2}$ and $v_{3,3}v_{2,1}$ are blue.
 With the same arguments we choose $f=v_{3,2}v_{2,3}$ accordingly.
 This shows that the red color class is isomorphic to $K_3\times K_3$.
\end{proof}

\begin{lemma}\label{lem::H96H12}
 $H_{9,6}\to H$ for each $H\in\{H_{3,2},H_1,H_2\}$.
\end{lemma}
\begin{proof}
Consider  a $2$-edge-coloring of $H_{9,6}$.   Let $v$ denote the vertex of degree $6$. We shall show that it contains each of the graphs from $\{H_{3,2},H_1,H_2\}$
as a monochromatic subgraph.

Either there is a  monochromatic copy of $H_{3,2}$ in the copy of $K_9$ in $H_{9,6}$ or we 
may assume by Lemma~\ref{lem::K9H32} that the $K_9$ in $H_{9,6}$ is an edge disjoint union of  a red $K_3 \times K_3$ and a blue  $K_3 \times K_3$.  Each edge in $K_9$ belongs to a unique  monochromatic triangle.
 If $v$ sends two blue edges  to vertices $u, w$, where $uw$ is blue, then  the   blue triangle containing $uw$  in $K_9$ together with $v$ form a blue $H_{3,2}$. 
 Thus, we may assume that neighborhood of $v$ via blue edges forms a red clique, and, similarly, its neighborhood via red edges forms a red clique.
 Since degree of $v$ is $6$, and the largest monochromatic clique in $K_9$ is a triangle,  these cliques must be triangles. 
However,  there are no two disjoint red and blue  triangles in $K_9$, so  we arrive at a contradiction. Thus, there is a monochromatic $H_{3,2}$.
 \medskip

Assume that the  monochromatic copy $K$ of $H_{3,2}$ is red.
First, we  assume that $K$ does not contain $v$. 
 Let $K^c$ denote the set of vertices from $K_9$ that are not in $K$.
 If there is a red edge between a vertex of degree $3$ of  $K$ and $K^c$, we have a monochromatic $H_1$.
 If there is a red edge between a vertex of degree $2$ of  $K$ and $K^c$, we have a monochromatic $H_2$.
 If all edges between degree $3$  vertices of $K$ and $K^c$ are blue and there are two adjacent blue edges in $K^c$,
 then there is a blue copy of $H_1$. 
 If all edges between degree $3$  vertices of $K$ and $K^c$ are blue and there are no two adjacent blue edges in $K^c$,
 then $K^c$  forms   a red $K_5$ minus a matching, and thus contains a copy of $H_1$.
 If all edges between degree $2$  vertices of $K$ and $K^c$ are blue, then there is a blue copy of $H_2$ or there is no blue edge induced by  $K^c$, 
 In the latter case $K_c$ induces a red $K_5$ that contains a red copy of $H_2$.
 
 Now, assume that any  monochromatic $H_{3,2}$  contains $v$, i.e.,  there is no monochromatic  
 $H_{3,2}$ in  a copy of $K_9$ of $H_{9,6}$.   Hence the coloring of $K_9$ is like it is described in Lemma~\ref{lem::K9H32}.  Then, it is easy to see that $K$ and  an appropriate  edge of $K_9$ form a  monochromatic copy of $H_1$ and, similarly, a monochromatic  copy of $H_2$.
\end{proof}

\end{appendix}

\end{document}